\documentclass[11pt]{amsart}
\usepackage{amssymb}
\usepackage{multirow}
\usepackage{mathtools}
\usepackage{hyperref}
\usepackage{tikz-cd}
\hypersetup{breaklinks=true}
\usepackage{breakurl}
\usepackage{graphicx}
\usepackage{tikz}

\usepackage[margin=1in]{geometry}

\newtheorem{thm}{Theorem}[section]
\newtheorem{cor}[thm]{Corollary}
\newtheorem{prop}[thm]{Proposition}
\newtheorem{lem}[thm]{Lemma}

\theoremstyle{definition}
\newtheorem{defn}[thm]{Definition}

\theoremstyle{remark}

\DeclareMathOperator{\Aut}{Aut}

\usepackage{color}

    \newtheoremstyle{TheoremNum}
        {7pt}{7pt}              %%% space between body and thm
        {\itshape}                      %%% Thm body font
        {}                              %%% Indent amount (empty = no indent)
        {\bfseries}                     %%% Thm head font
        {.}                             %%% Punctuation after thm head
        { }                             %%% Space after thm head
        {\thmname{#1}\thmnote{ \bfseries #3}}%%% Thm head spec
    \theoremstyle{TheoremNum}
    \newtheorem{thmn}{Theorem}

\begin{document}

\title{Geometry of Schreieder's varieties and some elliptic and K3 moduli curves}
\author{Laure Flapan}
\address{Department of Mathematics, Massachusetts Institute of Technology, Cambridge, MA 02139 USA}
\email{lflapan@mit.edu}

\begin{abstract}
We study the geometry of a class of $n$-dimensional smooth projective varieties constructed by Schreieder for their noteworthy Hodge-theoretic properties. In particular, we realize Schreieder's surfaces as elliptic modular surfaces and Schreieder's threefolds as one-dimensional families of Picard rank $19$ $K3$ surfaces. 
\end{abstract}
\maketitle

\section{Introduction}
If $X$ is a smooth complex projective variety of dimension $n$, then for any $k\le 2n$, the singular cohomology of $X$ comes equipped with a decomposition into complex subspaces given by $H^{k}(X,\mathbb{Q})\otimes \mathbb{C}\cong \bigoplus_{p+q=k} H^{p,q}(X)$, where $H^{p,q}(X)\cong H^q(X,\Omega_X^p)$. In \cite{schreieder}, Schreieder considered the question of whether any set of Hodge numbers $h^{p,q}\coloneqq \dim H^{p,q}(X)$, subject to the necessary symmetries imposed by the Hard Lefschetz Theorem, can be realized by a smooth complex projective variety. 

To this end, Schreieder \cite[Section 8]{schreieder}  constructed, among others, $n$-dimensional smooth projective varieties $X_{c,n}$, depending on a parameter $c\ge 1$, with particularly pathological Hodge numbers in that $X_{c,n}$ has positive $h^{n,0}=h^{0,n}$ and all other $h^{p,q}=0$ for $p\ne q$. The $X_{c,n}$ are smooth models of a quotient $C_g^n/G$, for $C_g$ a genus $g=\frac{3^c-1}{2}$ hyperelliptic curve and $G$ a finite group. 

Schreieder's construction generalized a construction of Cynk and Hulek in \cite[Section 3]{cynk} in the $c=1$ case, who proved that the $X_{1,n}$ are Calabi-Yau. From Cynk and Hulek's inductive construction, it follows that the $X_{1,n}$ may be realized as families of Calabi-Yau varieties over $\mathbb{P}^1$. Additionally, they proved that these Calabi-Yau varieties are modular, a result which was generalized to all of Schreieder's varieties in \cite[Corollary 3.8]{FlapanLang}.

In addition to having these noteworthy Hodge-theoretic and arithmetic properties, the varieties $X_{c,n}$ are also special from a cycle-theoretic point of view. Laterveer and Vial recently showed in \cite{LV} that the subring of the Chow ring of $X_{c,n}$ generated by divisors, Chern classes, and intersections of two positive-dimensional cycles injects into cohomology via the cycle class map. Moreover they show that in the surface case, the small diagonal of $Z_{c,2}$ admits a decomposition similar to that of K3 surfaces proved by Beauville-Voisin \cite{BV}. 

In this paper, we investigate in detail the geometry of the varieties $X_{c,n}$ for $c>1$. We generalize Cynk and Hulek's result for the $c=1$ case by showing in Proposition \ref{fibration} that for any $c\ge 2$, although $X_{c,n}$ has Kodaira dimension $1$ instead of $0$,   the Iitaka fibration of $X_{c,n}$ equips the variety $Z_{c,n}$, birational to $X_{c,n}$, with a fibration over $\mathbb{P}^1$ by hypersurfaces of Kodaira dimension $0$. 

In particular, in the cases of dimension $n=2$ and $n=3$, we obtain a moduli-theoretic interpretation of Schreieder's varieties via the following two main results. 
\begin{thmn}[\ref{elmod}] 
For $c\ge 2$, the minimal model $Z_{c,2}$ of $X_{c,2}$ is the elliptic modular surface attached to an explicit non-congruence subgroup $\Gamma_c\subset SL(2,\mathbb{Z})$.
\end{thmn}

\begin{thmn}[\ref{main theorem}] 
For $c\ge 2$, the general smooth fibers of the Iitaka fibration $Z_{c,3}\rightarrow \mathbb{P}^1$ associated to $X_{c,3}$ are K3 surfaces of Picard rank $19$. 
\end{thmn}

The notion of an elliptic modular surface is due to Shioda \cite{modsur}, who attaches to any finite index subgroup $\Gamma$ of $SL(2,\mathbb{Z})$ not containing $-\mathrm{Id}$ a corresponding extremal elliptic surface $S_\Gamma$. This $S_{\Gamma}$ is fibered over the modular curve $C_\Gamma$, given by  $\Gamma\backslash \mathcal{H}$ together with finitely many cusps, such that $S_{\Gamma}$ is a universal family for the moduli space of elliptic curves parametrized by the curve $C_\Gamma$. Therefore, Theorem \ref{elmod} implies that the surface $Z_{c,2}$ is a universal family for the moduli curve $C_{\Gamma_c}$. 

In terms of the threefold $Z_{c,3}$ fibered by K3 surfaces of Picard rank $19$, for $S$ a general smooth K3 fiber, consider the Neron-Severi group $NS(S)\coloneqq H^2(S,\mathbb{Z})\cap H^{1,1}(S)$ and the transcendental lattice $T_S=NS(S)^{\perp}$ in $H^2(S,\mathbb{Z})$.  Because $NS(S)$ has rank $19$, it follows from results of Morrison \cite{morrison2} and Nikulin \cite{nikulin} that there is a unique moduli curve parametrizing the K3 surfaces with this fixed transcendental lattice $T_S$. Hence, in analogy with our result in the two-dimensional case, Theorem \ref{main theorem} implies that the threefold $Z_{c,3}$ may be viewed as a finite cover of the universal family of the moduli curve parametrizing K3 surfaces with this transcendental lattice (see Corollary \ref{final cor}). 

More generally, it would be interesting to see if for any $n\ge 2$ the Iitaka fibration $Z_{c,n}\rightarrow \mathbb{P}^1$ is a finite cover of the universal family of the moduli curve parametrizing Calabi-Yau varieties with some fixed Hodge-theoretic data. However, the singularities of the quotient $C_g^n/G$ are non-canonical and thus understanding the geometry and Hodge theory of these fibers becomes difficult as $n$ grows. 

The organization of the paper is as follows. In Section \ref{construction} we outline Schreieder's construction in \cite{schreieder} of the varieties $X_{c,n}$. Then in Section \ref{kodaira dimension section} we show that that for $c\ge 2$ the varieties $X_{c,n}$ have Kodaira dimension $1$. In Section \ref{iitaka section} we analyze the geometry of the Iitaka fibration of $X_{c,n}$, proving that its image is the curve $\mathbb{P}^1$. In Section \ref{surface section}, we focus just on the $n=2$ case and analyze in detail the elliptic fibration resulting from the Iitaka fibration studied in Section \ref{iitaka section} and show that it equips $Z_{c,2}$ with the structure of an elliptic modular surface. Lastly in Section \ref{threefold section} we focus on the $n=3$ case, proving that the smooth fibers of the Iitaka fibration $Z_{c,3}\rightarrow \mathbb{P}^1$ are K3 surfaces of Picard rank $19$ and discussing the moduli interpretations of this result.

%%%%%%%%%%%%%%%%%%%%%%%%%%%%%%%%%%%%%%%%%%%%%%%%%%%%%%%%%%%%%%%%%%%
%%%%%%%%%%%%%%%%%%%%%%%%%%%%%%%%%%%%%%%%%%%%%%%%%%%%%%%%%%%%%%

\section{Construction of $X_{c,n}$}\label{construction}
For a fixed $c\ge2$ consider the complex hyperelliptic curve $C_g$ of genus $g=\frac{3^c-1}{2}$ given by the smooth projective model of the affine curve 
\[\{y^2=x^{2g+1}+1\}.\]
obtained by adding a point at $\infty$. This point is covered by an affine piece
$\{v^2=u^{2g+2}+u\},$
such that  $x=u^{-1}$ and $y=v\cdot u^{-g-1}$. One may verify that $x$ is a local coordinate in the patch $\{y^2=x^{2g+1}+1\}$ and that $v$ is a local coordinate in the patch $\{v^2=u^{2g+2}+u\}.$

Fix $\zeta$ a primitive $3^c$-th root of unity. The curve $C_g$ then comes equipped with an automorphism $\psi_g$ of order $3^c=2g+1$ given by
\begin{align*}
(x,y)&\mapsto (\zeta x, y)\\
(u,v)&\mapsto (\zeta^{-1} u, \zeta^{g} v).
\end{align*}

Consider the action on the $n$-dimensional product $C_g^n$ given by the group
\[G\coloneqq \{\psi_g^{a_1}\times \cdots \times \psi_g^{a_n}\mid a_1+\cdots+a_n\equiv 0 \bmod 3^c\},\]
where the automorphism $\psi_g^{a_i}$ acts on the $i$-th factor in the product. Note that $G\cong \left(\mathbb{Z}/3^c\mathbb{Z}\right)^{n-1}$.

 The smooth projective variety $X_{c,n}$ whose construction we detail below is then a smooth model of the quotient $C_g^n/G$. This $n$-dimensional variety $X_{c,n}$ is constructed inductively as a pair $(X_{c,n}, \phi_{c,n})$, where $\phi_{c,n}$ is a distinguished automorphism of $X_{c,n}$.

The inductive construction of $(X_{c,n}, \phi_{c,n})$ detailed in \cite[Section 8]{schreieder} proceeds as follows.  Suppose the pairs of varieties and distinguished automorphisms $(X_{c,n_1}, \phi_{n_1})$ and $(X_{c,n_2}, \phi_{n_2})$ have both been constructed. We then detail the construction of the pair $(X_{c,n_1+n_2},\phi_{c,n_1+n_2})$.

Consider the subgroup of $\Aut(X_{c,n_1}\times X_{c,n_2})$ given by
\[
H:= \langle \phi_{c,n_1}^{-1}\times \mathrm{id}, \mathrm{id}\times \phi_{c,n_2}\rangle.
\] 
For each $i=1,\ldots, c$, consider the element of order $3^i$ inside of $H$ given by
\[\eta_i:=( \phi_{c,n_1}^{-1}\times \phi_{c,n_2})^{3^{c-i}}.\]
Thus $\eta_i$ generates a cyclic subgroup $H_i:=\langle\eta_i\rangle \subset H,$ which gives a filtration
\[0=H_0\subset H_1\subset \cdots \subset H_c=\langle \phi_{c,n_1}^{-1}\times \phi_{c,n_2}\rangle\]
such that each quotient $H_i/H_{i-1}$ is cyclic of order $3$.  

Now, let
\begin{equation*}
\begin{aligned}
Y_0&=X_{c,n_1}\times X_{c,n_2},\\
Y_0'& = \text{Blow up of } Y_0 \text{ along } \mathrm{Fix}_{Y_0}(\eta_1),\\
Y_0''& =\text{Blow up of } Y_0' \text{ along } \mathrm{Fix}_{Y_0'}(\eta_1).
\end{aligned}
\end{equation*}
Observe that since the action of the group $H$ restricts to an action on $ \mathrm{Fix}_{Y_0}(\eta_1)$, the action of $H$ on $Y_0$ lifts to an action on $Y_0'$ and then similarly to an action on $Y_0''$. Here, by abuse of notation, we let $\langle\eta_1\rangle$ denote both the subgroups of $\mathrm{Aut}(Y_{0}')$ and $\mathrm{Aut}(Y_{0}'')$ generated by the action of $\eta_1\in H$.

Define for $i\in \{1,\ldots,c\}$: 
\begin{equation*}
\begin{aligned}
Y_i&=Y_{i-1}''/\langle\eta_i\rangle,\\
Y_i'& = \text{Blow up of } Y_i \text{ along } \mathrm{Fix}_{Y_i}(\eta_{i+1}),\\
Y_i''& =\text{Blow up of } Y_i' \text{ along } \mathrm{Fix}_{Y_i'}(\eta_{i+1}).
\end{aligned}
\end{equation*}
Namely we have the following diagram
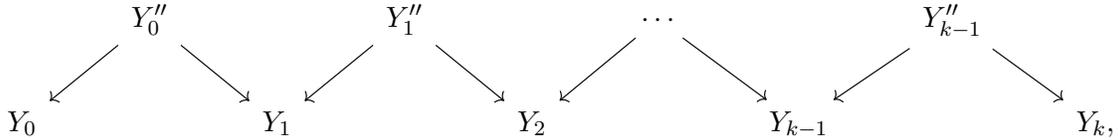
\[
\begin{tikzcd}
&\arrow{dl} Y_0''\arrow{dr}&& \arrow{dl}Y_1''\arrow{dr}&&\arrow{dl}\cdots \arrow{dr} && \arrow{dl}Y_{k-1}''\arrow{dr}&\\
Y_0 && Y_1&& Y_2&& Y_{k-1}&&Y_k,
\end{tikzcd}
\]
where each arrow to the left in the above diagram corresponds to a sequence of two blow-up maps and each arrow to the right corresponds to a $3:1$ cover.

Schreieder proves in \cite[Proposition 19] {schreieder} that each $Y_{i+1}$ is a smooth model of $Y_{i}/\langle\eta_i\rangle$ and thus, in particular, the variety $X_{c,n_1+n_2}\coloneqq Y_c$ is a smooth model of $Y_0/\langle \phi_{c,n_1}^{-1}\times \phi_{c,n_2}\rangle$. The distinguished automorphism $\phi_{c,n_1+n_2}$ on $X_{c,n_1+n_2}$ is then defined to be the one induced by the automorphism $\mathrm{id}\times \phi_{c,n_2}$ on $Y_0$. 

The inductive construction of the pairs $(X_{c,n}, \phi_{c,n})$ is then as follows. When $n=1$, the pair $(X_{c,1}, \phi_{c,1})$ is the curve $X_{c,1}\coloneqq C_g$ equipped with the distinguished automorphism $\phi_{c,1}\coloneqq \psi_g$. Hence, by the above construction, if one can construct the pair $(X_{c,n},\phi_{c,n})$, one can construct the pair $(X_{c,n+1},\phi_{c,n+1})$. Namely one can construct $(X_{c,n}, \phi_{c,n})$ for all $n\ge 1$.

%%%%%%%%%%%%%%%%%%%%%%%%%%%%%%%%%%%%%%%%%%%%%%%%%%%%%%%%%%%
%%%%%%%%%%%%%%%%%%%%%%%%%%%%%%%%%%%%%%%%%%%%%%%%%%%%%%%%%%%
\section{The Kodaira dimension of $X_{c,n}$}\label{kodaira dimension section}

For a smooth algebraic variety $V$ and any $m> 0$, the $m$\emph{-th plurigenus} of $V$ is given by $P_m=h^0(V,K_{V}^{\otimes m})$. The \emph{Kodaira dimension} $\kappa$ of $V$ is $-\infty$ if $P_m=0$ for all $m>0$ and otherwise it is the minimum $\kappa$ such that $P_m/m^\kappa$ is bounded. If $V$ has dimension $n$, then the Kodaira dimension of $V$ is either $-\infty$ or an integer $0\le \kappa \le n$.

In order to compute the Kodaira dimension of the variety $X_{c,n}$ for $c\ge 2$, we thus wish to consider the plurigenera $P_m=h^0(X_{c,n},K_{X_{c,n}}^{\otimes m})$ for $m>0$. We show in Proposition \ref{koddim}, by inducting on the dimension $n$, that the plurigenera $P_m$ grow linearly with $m$ and hence that the varieties $X_{c,n}$ for $c\ge 2$ have Kodaira dimension $1$. 

As detailed in Section \ref{construction}, the variety $X_{c,n}$ is constructed from $C_g^n$ by a sequence of blow-ups and quotients. We thus have an injection of global sections of powers of the canonical bundle induced by the rational map $C_g^n\dashrightarrow X_{c,n}$. To compute the $P_m$, one then just needs to determine which of the $G$-invariant global sections of $K_{C_g^n}^{\otimes m}$ descend to global sections of $K_{X_{c,n}}^{\otimes m}$. To accomplish this, it is necessary first in Sections \ref{quotient} and \ref{blowup} to establish what happens to global sections of powers of the canonical bundle of $C_g^n$ under blow-ups and quotients. 

It is also necessary to understand the local action of the group $G$ and in particular the relationship between the local action in the $n$-dimensional case in relation to the $n+1$-dimensional case. Thus in Section \ref{phi action section} we analyze the local action of the automorphism $\phi_{c,n}$ on $X_{c,n}$. We then use this in Sections \ref{local weight section} and \ref{vanishing section} to identify distinguished coordinate patches $R_c$ and $S_c$ on $X_{c,n}$ on which we may describe the vanishing of a form in $H^0(X_{c,n},K_{X_{c,n}}^{\otimes m})$ in relation to the vanishing of the corresponding form on $Y_0\coloneqq C_g\times X_{c,n-1}$. 

The ingredients of Sections \ref{quotient}--\ref{vanishing section} are then used in the proof of Proposition \ref{koddim} to trace the vanishing of forms corresponding to global sections of powers of the canonical bundle of $C_g^n$ through the construction detailed in Section \ref{construction}. We determine that the only global sections of $K_{C_g^n}^{\otimes m}$ that descend to global sections of $K_{X_{c,n}}^{\otimes m}$ are those of the form $x_1^a\omega_1\times x_2^a\omega_2\times \cdots \times x_n^a\omega_n$ for $0\le a\le m(g-1)$, where $(x_i,y_i)$ are coordinates on the $i$-th factor in the product $C_g^n$ and $\omega_i=\frac{dx_i^{\otimes m}}{y_i^m}$, and hence the Kodaira dimension of $X_{c,n}$ is $1$. 

%%%%%%%%%%%%%%%%%%%%%%%%%%%%%%%%%%%%%%%%%%%%%%%%%%%%%%

\subsection{Forms Under Quotients}\label{quotient}
Recalling the notation from the construction of $X_{c,n}$ in Section \ref{construction}, consider the $3:1$ cover maps $f_i: Y_i''\rightarrow Y_{i+1}$, where $Y_i''$ and $Y_{i+1}$ have dimension $n$. The Riemann-Hurwitz formula yields
\begin{equation}\label{rhurwitz}K_{Y_i''}=f_i^*\left(K_{Y_{i+1}}+\sum_{D\in \mathrm{Div}(Y_{i+1})}\frac{a_D-1}{a_D}D\right)\end{equation}
where for each $D\in \mathrm{Div}(Y_{i+1})$, the number $a_D$ is the order of the group of automorphisms of $Y_{i}''$ fixing the components of $f_i^*D$ pointwise. 

By construction, the group $H_{i+1}/H_i$ acting on $Y_i''$ is isomorphic to $\mathbb{Z}/3\mathbb{Z}$. Namely for every irreducible divisor $D\in \mathrm{Div}(Y_{i+1})$, either $a_D=1$ or $a_D=3$. Moreover, the irreducible $D$ for which $a=3$ are exactly the images of the irreducible components of the exceptional divisors $E_i''$ obtained from the blow-up map $Y_i''\rightarrow Y_i'$, where it should be noted that it may happen that $E_i''\cong E_i'$. Let $E_{i,1}''\ldots, E_{i,k_i}''$ be the irreducible components of  $E_i''$. Observe that since $\eta_{i+1}$ fixes each of the $E_{i,j}''$, each component $E_{i,j}''$ descends to an irreducible divisor on $Y_{i+1}$. Equation (\ref{rhurwitz}) then yields:
\[K_{Y_i''}=f_i^*\left(K_{Y_{i+1}}\right)+\sum_{j=1}^{k_i}2E_{i,j}''.\]
This gives:
\begin{equation}\label{rhurwitz2}K_{Y_i''}^{\otimes m}-\sum_{j=1}^{k_i}2mE_{i,j}'' =f_i^*\left(K_{Y_{i+1}}^{\otimes m}\right).\end{equation}

For an algebraic variety $V$ with a coordinate patch $(z_1,\ldots,z_n)$ having the standard action of $\mathbb{G}_m^n$ on $\mathbb{C}^n$, we say that a pluriform $\omega$ is \emph{toric} on the patch $(z_1,\ldots,z_n)$ if the divisor of zeros of $\omega$ is invariant under the action of $\mathbb{G}_m^n$. 

\begin{defn} A toric form $\omega$ on a coordinate patch $(z_1,\ldots,z_n)$ of an algebraic variety $V$ has \emph{vanishing sequence} $(\beta_1,\ldots, \beta_n)$ on the point $(z_1,\ldots,z_n)=(0,\ldots,0)$ if $\omega$ vanishes to order $\beta_i$ along the hypersurface $z_i=0$.
\end{defn}
Now consider an $H_{i+1}$-invariant pluriform $\sigma$ on $Y_i''$.  Suppose $Y_i''$ has local coordinates $(z_1,\ldots,z_n)$ around some $E_{i,j}''$ fixed by the action of $H_{i+1}$ such that without loss of generality $E_{i,j}''$ is given by $z_1=0$. Consider the point  $R=(0,\ldots,0)$ on $E_{i,j}''$ and suppose the vanishing sequence of $\sigma$ on $R$ is
$(\alpha_1,\ldots,\alpha_n).$
Then, using Equation (\ref{rhurwitz2}), the vanishing sequence of the descent of $\sigma$ to $Y_{i+1}$ has vanishing sequence on the image of $R$ in $Y_{i+1}$ given by
\begin{equation}\label{quotienteq}(\frac{1}{3}(\alpha_1-2m), \alpha_2,\ldots,\alpha_n).\end{equation}

%%%%%%%%%%%%%%%%%%%%%%%%%%%%%%%%%%%%%%%%%%%%%%%%%%%%%%
\subsection{Forms Under Blow-Ups}\label{blowup}

Let $V$ be an $n$-dimensional variety with local coordinates $(z_1,\ldots, z_n)$. Suppose $Z$ is a subvariety of $V$ of codimension $k\ge 2$ given locally by $z_1=\cdots=z_k=0$. Suppose $\sigma$ is a global section of $K_V^{\otimes m}$ for $m\ge 1$ with vanishing sequence  $(\alpha_1,\ldots, \alpha_k,0,\ldots,0)$ on $Z$. Namely, the pluriform $\sigma$ is given locally by
\[f(z_1, \ldots, z_n)(dz_1\cdots dz_n)^{\otimes m},\]
where the polynomial $f$ has vanishing sequence $(\alpha_1,\ldots,\alpha_k,0,\ldots,0)$ on $z_1=\cdots=z_k=0$. 

Blowing up $V$ at $Z$ introduces new coordinates $(z_1',\ldots, z_k')$, with $z_iz_j'=z_jz_i'$ for all $i, j$. Hence, on the coordinate patch of the blown-up variety  $V'$ given by $z_i'\ne0$, we have coordinates
\[(\tilde{z}_1,\ldots, \tilde{z}_{i-1},z_i, \tilde{z}_{i+1}, \ldots, \tilde{z}_k,z_{k+1},\ldots, z_n),\]
 where $\tilde{z}_j=\frac{z_j'}{z_i'}$ and thus $z_j=z_i\tilde{z}_j$.  Thus, locally around the exceptional divisor $E$, the pluriform $\sigma$ pulls back to the pluriform:
\[z_i^{m(k-1)}f(z_i\tilde{z}_1, \ldots, z_i\tilde{z}_{i-1}, z_i, z_i\tilde{z}_{i+1},\ldots, z_i\tilde{z}_k, z_{k+1}, \ldots, z_n)(d\tilde{z}_1\cdots d\tilde{z}_{i-1}dz_id\tilde{z}_{i+1}\cdots d\tilde{z}_kdz_{k+1}\cdots dz_n)^{\otimes m}.\]

In the new coordinates $(\tilde{z}_1,\ldots, \tilde{z}_{i-1},z_i, \tilde{z}_{i+1}, \ldots, \tilde{z}_k,z_{k+1},\ldots, z_n),$ consider the point $R=(0,\ldots,0)$. Then the vanishing sequence on $R$ of the pullback of $\sigma$ to $V'$ is given by:
\begin{equation}\label{blowupeq}(\alpha_1,\ldots,\alpha_{i-1},\sum_{l=1}^k \alpha_l +m(k-1), \alpha_{i+1},\ldots, \alpha_n).\end{equation}

%%%%%%%%%%%%%%%%%%%%%%%%%%%%%%%%%%%%%%%%%%%%%%%%%%%%%%

\subsection{Local Action of $\phi_{c,n}$}\label{phi action section}
Consider the curve $C_g$ together with its action by the automorphism $\psi_g$. Observe that the fixed set of the action of $\psi_g$ on $C_g$ consists of $3$ points. Indeed, in the coordinate patch given by  $\{y^2=x^{2g+1}+1\}$ we have two fixed points
\[P_1\colon (x,y)=(0,1) \text{ and } P_2\colon (x,y)=(0,-1)\]
and in the coordinate patch given by $\{v^2=u^{2g+2}+u\}$, we have a third fixed point
\[Q\colon (u,v)=(0,0).\]
Now recall the construction of the variety $X_{c,n}$ and its distinguished automorphism $\phi_{c,n}$ from the curve $C_g$ and its automorphism $\psi_g$ detailed in Section \ref{construction}. In the following lemma, we detail the local action of $\phi_{c,n}$ and the corresponding vanishing of forms on two particular coordinate patches on $X_{c,n}$.

\begin{lem}\label{localaction}
For any $c\ge 2$ and $n\ge 1$, there exist coordinate patches $U_{n}$ and $V_n$ on $\mathrm{Fix}_{X_{c,n}}(\phi_{c,n})$, with local coordinates $(u_1,\ldots, u_n)$ and $(v_1,\ldots, v_n)$ respectively, satisfying:
\begin{enumerate}
\item The automorphism $\phi_{c,n}$ acts with $\mathbb{Z}/3^c\mathbb{Z}$-weights $(0,\ldots, 0, 1)$ on $U_n$ and  $(0,\ldots, 0, g)$ on $V_n$
\item For any $\tau\in H^0(C_g^{n-1}, K_{C_g^{n-1}}^{\otimes m})$ and $\sigma\in H^0(C_g, K_{C_g}^{\otimes m})$ such that $\sigma$ vanishes to order $\alpha$ on the point $P_1$ and to order $\beta$ on $Q$, the global section of $K_{X_{c,n}}^{\otimes m}$  induced by the form $\tau\times \sigma$ vanishes to order $\alpha$ on the hypersurface $u_n=0$ and to order $\beta$ on the hypersurface $v_n=0$
\end{enumerate}
\end{lem}

\begin{proof}
We proceed by induction on the dimension $n$. By the implicit function theorem, one may verify that the coordinate $x$ is a local coordinate in the patch $U_1$ on the curve $C_g$ given by  $\{y^2=x^{2g+1}+1\}$ and the coordinate $v$ is a local coordinate in the coordinate patch $V_1$ given by $\{v^2=u^{2g+2}+u\}$. Hence the automorphism $\phi_{c,1}\coloneqq \psi_g$ acts with $\mathbb{Z}/3^c\mathbb{Z}$-weight $1$ on $U_1$ and weight $g$ on $V_1$. Since $\mathrm{Fix}_{X_{c,1}}(\phi_{c,1})=\{P_1,P_2,Q\}$, this verifies the statement of the lemma in the $n=1$ case. 

Assume the result holds in the $n-1$ case and let $U_{n-1}$ and $V_{n-1}$ be the corresponding coordinate patches on $\mathrm{Fix}_{X_{c,n-1}}(\phi_{c,n-1})$. Now recall the inductive construction of the pair $(X_{c,n},\phi_{c,n})$ from the pair $(X_{c,n-1},\phi_{c,n-1})$ detailed in Section \ref{construction}. Let $Y_0=C_g\times X_{c,n-1}$ and consider the patches on $\mathrm{Fix}_{Y_0}(\psi_g\times \phi_{c,n-1})$ given by 
\[R_0\coloneqq U_1\times V_{n-1} \text{ \ and \ } S_0\coloneqq V_1\times U_{n-1}.\]

By assumption the automorphism $\eta_c\coloneqq \psi_g^{-1}\times \phi_{c,n-1}$ acts on these patches respectively with $\mathbb{Z}/3^c\mathbb{Z}$-weights 
\[(-1,0,\ldots, 0, g) \text{\ and \ } (-g,0,\ldots,0,1)\]
and the automorphism $\mathrm{id}\times \phi_{c,n-1}$ acts respectively with $\mathbb{Z}/3^c\mathbb{Z}$-weights 
\[(0,\ldots,0,g)\text{ \ and \ }(0,\ldots,0,1).\] 

Now consider the images of $R_0$ and $S_0$ along the sequence of blow-ups and quotients detailed in Section \ref{construction}. Observe that because $\eta_1$ acts with weight $0$ in all but two of the local coordinates around $R_0$ and $S_0$, for each $i\in \{1,\ldots,c\}$ the fixed locus of $\eta_i$ on the images of $R_0$ and $S_0$ will have codimension at most $2$. 

When the codimension is less then $2$, then the blow-up map is an isomorphism and so the local weights are unaffected. When the codimension is exactly $2$, the blow-up map locally introduces new coordinates which we denote $(\tilde{r}_1,\tilde{r}_n)$ and $(\tilde{s}_1,\tilde{s}_n)$ respectively. 

Inductively, taking the $\tilde{r}_1\ne0$ and $\tilde{s}_1\ne 0$ patches in the new blown-up coordinates ensures that $\mathrm{id}\times \phi_{c,n-1}$ still acts with weights $(0,\ldots,0,g)$  and $(0,\ldots,0,1)$ on these new patches. Moreover, the $\mathbb{Z}/3^c\mathbb{Z}$-weights of $\mathrm{id}\times \phi_{c,n-1}$ on the image of these coordinate patches under the $3:1$ quotient maps will be unaffected since neither $g$ nor $1$ is divisible by $3$. 

Hence let $V_n$ and $U_n$ be the coordinate patches on $Y_c\coloneqq X_{c,n}$ obtained by locally choosing the $\tilde{r}_1\ne0$ and $\tilde{s}_1\ne 0$ patches respectively at each stage in the sequence of blowups.  It then follows that the image of $\mathrm{id}\times \phi_{c,n-1}$ acts with $\mathbb{Z}/3^c\mathbb{Z}$-weights $(0,\ldots,0,g)$ on $V_n$ and $(0,\ldots,0,1)$ on $U_n$. Since $\phi_{c,n}$ is exactly the image of $\mathrm{id}\times \phi_{c,n-1}$  in $X_{c,n}$, this finishes the first part of the proof. The second part follows from the formulas in equations \eqref{quotienteq} and \eqref{blowupeq}.
\end{proof}

%%%%%%%%%%%%%%%%%%%%%%%%%%%%%%%%%%%%%%%%%%%%%%%%%%%%%%
\subsection{The patches $R_c$ and $S_c$ on $X_{c,n}$}\label{local weight section}
Consider the product $Y_0\coloneqq C_g\times X_{c,n-1}$ together with its actions by the automorphisms $\eta_i\coloneqq (\psi_g^{-1}\times \phi_{c,n-1})^{3^{c-i}}$. In particular, let us consider the action of $\eta_c=\psi_g^{-1}\times \phi_{c,n-1}$. We now use Lemma \ref{localaction} to identify two distinguished patches on $X_{c,n}$, which we will denote by $R_c$ and $S_c$. 

Let $U_{n-1}$ and $V_{n-1}$ be the coordinate patches on $\mathrm{Fix}_{X_{c,n-1}}(\phi_{c,n-1})$ and $U_1$ and $V_1$ be the coordinate patches on $C_g$ determined in Lemma \ref{localaction}. It follows that the automorphism $\eta_c$ acts with $\mathbb{Z}/3^c\mathbb{Z}$-weights $(-1,0,\ldots,0,g)$ on $R_0\coloneqq U_1\times V_{n-1}$ and with $\mathbb{Z}/3^c\mathbb{Z}$-weights $(-g,0,\ldots,0,1)$ on $S_0\coloneqq V_1\times U_{n-1}$.

The blow-up map $Y_0'\rightarrow Y_0$ introduces new local coordinates $(\tilde{r}_1,\tilde{r}_n)$ and $(\tilde{s}_1,\tilde{s}_n)$ respectively on $R_0$ and $S_0$ and $\eta_c$ acts with weights 
\[(-(g+1),0,\ldots, 0, g)\]
 on the $\tilde{r}_n\ne0$ patch $R_0'$ of the strict transform of $R_0$ and with weights 
 \[(-g,0,\ldots,0,g+1)\]
 on the $\tilde{s}_1\ne0$ patch $S_0'$ of the strict transform of $S_0$. 
Similarly, the blow-up map $Y_0''\rightarrow Y_0'$ introduces new local coordinates $(\tilde{\tilde{r}}_1,\tilde{\tilde{r}}_n)$ and $(\tilde{\tilde{s}}_1,\tilde{\tilde{s}}_n)$ such that $\eta_c$ acts with weights 
\[(-(2g+1),0,\ldots, 0, g)=(0,0,\ldots, 0, g)\]
on the $\tilde{\tilde{r}}_n\ne0$ patch $R_0''$ of the strict transform of $R_0'$ and with weights 
\[(-g,0,\ldots,0,2g+1)=(-g,0,\ldots,0,0)\]
 on the $\tilde{\tilde{s}}_1\ne0$ patch $S_0''$ of the strict transform of $S_0'$. 

Let $R_1$ and $S_1$ be the images of $R_0''$ and $S_0''$ under the $3:1$ quotient map $Y_0''\rightarrow Y_1$. Then $\eta_c$ acts on $R_1$ with $\mathbb{Z}/3^{c-1}\mathbb{Z}$-weights $(0,0,\ldots, 0, g)$ and on $S_1$ with $\mathbb{Z}/3^{c-1}\mathbb{Z}$-weights $(-g,0,\ldots,0,0)$. 

In particular, the fixed locus under this $\eta_c$-action has codimension $1$ in both these patches. Namely the blow-up maps $Y_1''\rightarrow Y_1'\rightarrow Y_1$ are isomorphisms on $R_1$ and $S_1$. Inductively defining $R_i$ and $S_i$ to be the images under the $3:1$ quotient map $Y_{i-1}''\rightarrow Y_i$ of $R_{i-1}$ and $S_{i-1}$ respectively, we then have that in fact the blow-up maps $Y_i''\rightarrow Y_i'\rightarrow Y_i$ are all isomorphisms on the coordinate patches $R_i$ and $S_i$. 

Namely, the coordinate patch $R_c$ is obtained from $R_1$ simply by performing a sequence of $c-1$ quotients by $\mathbb{Z}/3\mathbb{Z}$ and the coordinate patch $S_c$ is similarly obtained from $S_1$ by performing a sequence of $c-1$ quotients by $\mathbb{Z}/3\mathbb{Z}$.
%%%%%%%%%%%%%%%%%%%%%%%%%%%%%%%%%%%%%%%%%%%%%%%%%%%%%%

\subsection{Vanishing of Forms on $R_c$ and $S_c$}\label{vanishing section}
In the notation of Section \ref{local weight section}, consider a form $\sigma$ with vanishing sequence $(\alpha_1,\ldots,\alpha_n)$ on the point $(0,\ldots,0)$ of the coordinate patch $R_0$ and vanishing sequence $(\beta_1,\ldots,\beta_n)$ on the point $(0,\ldots,0)$ of the coordinate patch $S_0$. 

By Equation \ref{blowupeq}, the form $\sigma$ has vanishing sequences at the origins of $R_0'$ and $S_0'$ respectively given by
\[(\alpha_1,\ldots,\alpha_{n-1}, \alpha_1+\alpha_n+m) \text{ \ \  and \ \  }(\beta_1+\beta_n+m,\beta_2,\ldots, \beta_n)\]
and vanishing sequences at the origins of $R_0''$ and $S_0''$ respectively given by
\[(\alpha_1,\ldots,\alpha_{n-1}, 2\alpha_1+\alpha_n+2m) \text{ \ \  and \ \  }(\beta_1+2\beta_n+2m,\beta_2,\ldots, \beta_n).\]

By Equation \ref{quotienteq}, the form $\sigma$ then has vanishing sequences at the origins of $R_1$ and $S_1$ respectively given by
\[(\alpha_1,\ldots,\alpha_{n-1}, \frac{1}{3}(2\alpha_1+\alpha_n))\text{ \ \  and \ \  }(\frac{1}{3}(\beta_1+2\beta_n),\beta_2,\ldots, \beta_n).
\]

But as established in Section \ref{local weight section}, the coordinate patch $R_c$ is obtained from $R_1$ simply by performing a sequence of $c-1$ quotients by $\mathbb{Z}/3\mathbb{Z}$ and the coordinate patch $S_c$ is similarly obtained from $S_1$ by performing a sequence of $c-1$ quotients by $\mathbb{Z}/3\mathbb{Z}$. Hence it follows that $\sigma$ has vanishing sequence at the origin in $R_c$ given by
\begin{equation}\label{vanishing R_c}
\left(\alpha_1,\ldots,\alpha_{n-1},\frac{1}{3}\big(\cdots\big(\frac{1}{3}\big(\frac{1}{3}(2\alpha_1+\alpha_n)-2m\big)\cdots\big)-2m\big)\right)=\left(\alpha_1,\ldots,\alpha_{n-1},\frac{1}{3^c}(2\alpha_1+\alpha_n-m(3^c-3))\right)
\end{equation}
and has vanishing sequence at the origin in $S_c$ given by
\begin{equation}\label{vanishing S_c}
\left(\frac{1}{3}\big(\cdots\big(\frac{1}{3}\big(\frac{1}{3}(\beta_1+2\beta_n)-2m\big)\cdots\big)-2m\big),\beta_2,\ldots,\beta_n\right)=\left(\frac{1}{3^c}(\beta_1+2\beta_n-m(3^c-3)),\beta_2,\ldots,\beta_n\right).
\end{equation}

%%%%%%%%%%%%%%%%%%%%%%%%%%%%%%%%%%%%%%%%%%%%%%%%%%%%%%

\subsection{Kodaira dimension Computation for $X_{c,n}$}

We now make use of what we have established in Sections \ref{quotient}--\ref{vanishing section} to prove in Proposition \ref{koddim} that the varieties $X_{c,n}$ have Kodaira dimension $1$. To do this we use the following theorem of K\"{o}ck and Tait.
\begin{thm}\label{global}\cite[Theorem 5.1]{kock} Let $C$ be a hyperelliptic curve of genus $g\ge 2$ of the form $y^2=f(x)$ and let $\omega\in K_C^{\otimes m}$ be given by $\omega=\frac{dx^{\otimes m}}{y^m}$. Then an explicit basis for $H^0(C,K_C^{\otimes m})$ is given by the following:
$$\begin{cases}
\omega, x\omega, \ldots, x^{g-1}\omega & \mbox{if } m=1\\
\omega, x\omega, x^2\omega & \mbox{if }m=2 \mbox{ and } g=2\\
\omega, x\omega, \ldots, x^{m(g-1)}\omega; y\omega, xy\omega, \ldots, x^{(m-1)(g-1)-2}y\omega & \mbox{otherwise}
\end{cases}.$$
\end{thm}

\begin{prop}
\label{koddim} For any $c\ge2$ and $n\ge 2$, the  variety $X_{c,n}$ has Kodaira dimension $1$.\end{prop}

\begin{proof}
Fix some $m>0$ and consider the form on the affine patch $U_1\coloneqq \{y^2=x^{2g+1}+1\}$ of the curve $C_g$ given by
\[\omega\coloneqq \frac{dx^{\otimes m}}{y^{m}}.\]
By Theorem \ref{global}, we are interested in global sections of $K_{C_g}^{\otimes m}$ of the form $x^a\omega$, where $0\le a\le m(g-1)$ or of the form $x^ay\omega$, where $0\le a \le (m-1)(g-1)-2$. 

Note that since  the variable $x$ is a local coordinate near the $\psi_g$-fixed points $P_1$ and $P_2$ of $C_g$,  the form $\omega$ has order of vanishing equal to $0$ at $P_1$ and $P_2$. Hence both $x^{a}\omega$ and $x^{a}y\omega$ have order of vanishing equal to $a$ at the fixed points $P_1$ and $P_2$. 

On the affine patch $V_1\coloneqq \{v^2=u^{2g+2}+u\}$, the variable $v$ is a local coordinate near the $\psi_g$-fixed point $Q$ and the form $\omega$ is given by
\[\frac{(-1)^{d}u^{m(g-1)}du^{\otimes m}}{v^{m}}\]
The equation $v^2=u^{2g+2}+u$ yields $2v\cdot dv=((2g+2)u^{2g+1}+1)\cdot du$. Namely, $du$ and $v$ vanish to the same order. Moreover, $u$ has order of vanishing $2$ with respect to $v$, hence the order of vanishing of $\omega$ at the point $Q$ is $2m(g-1)=m(3^{c}-3)$. Thus a form $x^{a}\omega=u^{-a}\omega$ has order of vanishing at $Q$ given by 
\[m(3^{c}-3)-2a\]
and a form $x^{a}y\omega=u^{-(a+g+1)}v\omega$ has order of vanishing at $Q$ given by
\[2m(g-1)-2(a+g+1)+1=m(3^{c}-3)-2a-3^c.\]

Since global sections of $K_{X_{c,n}}^{\otimes m}$ inject into global sections of $K_{C_g^n}^{\otimes m}$ via the rational map $C_g^n\dashrightarrow X_{c,n}$, in order to prove that the varieties $X_{c,n}$ have Kodaira dimension $1$, it is enough to show that global sections of $K_{X_{c,n}}^{\otimes m}$ correspond to global sections of $K_{C_g^n}^{\otimes m}$ of the form $x_1^a\omega_1\times x_2^a\omega_2\times \cdots \times x_n^a\omega_n$ for $0\le a\le m(g-1)$. Here we implicitly use the natural map $H^0(C_g, K_{C_g})^{\otimes n} \rightarrow H^0(C_g^n,K_{C_g^n}^{\otimes m})$.

More precisely we will show that any global section of $K_{C_g^n}^{\otimes m}$ which descends to a global section of $K_{X_{c,n}}^{\otimes m}$ must be of the form $x_1^a\omega_1\times x_2^a\omega_2\times \cdots \times x_n^a\omega_n$ for $0\le a\le m(g-1)$. Since the number of such forms is linear in $m$, the Kodaira dimension of $X_{c,n}$ is at most $1$. However, by construction, we have $h^0(X_{c,n},K_{X_{c,n}})=g$, where we know $g\ge4$.  In particular, this means $P_1=h^0(X_{c,n},K_{X_{c,n}})$ is greater than $1$, so the Kodaira dimension of $X_{c,n}$ is at least equal to $1$. Hence, the Kodaira dimension of $X_{c,n}$ is exactly equal to $1$.

In order to prove that any global section of $K_{C_g^n}^{\otimes m}$ which descends to a global section of $K_{X_{c,n}}^{\otimes m}$ must be of the form $x_1^a\omega_1\times x_2^a\omega_2\times \cdots \times x_n^a\omega_n$ for $0\le a\le m(g-1)$, we will induct on the dimension $n$ of the variety $X_{c,n}$.

So let us begin with the $n=2$ case and consider $Y_0=C_g\times C_g$. Let $\sigma$ be a global section of $K_{Y_0}^{\otimes m}=K_{C_g^2}^{\otimes m}$ with vanishing sequence $(\alpha_1,\alpha_2)$ on the point $(0,0)=(P_i,Q)$ of $R_0$ and with vanishing sequence $(\beta_1,\beta_2)$ on the points $(0,0)=(Q, P_i)$ of $S_0$, where $R_0$ and $S_0$ are the coordinate patches of $Y_0$ defined in Section \ref{local weight section}. Then by Equations \eqref{vanishing R_c} and \eqref{vanishing S_c}, the form $\sigma$ has vanishing sequences 
\[\left(\alpha_1,\frac{1}{3^c}(2\alpha_1+\alpha_2-m(3^c-3)\right) \ \text{ and }\  \left(\frac{1}{3^c}(\beta_1+2\beta_2-m(3^c-3)),\beta_2\right)\]
at the origins of the patches $R_c$ and $S_c$ respectively in $X_{c,2}$. Therefore, if $\sigma$ corresponds to a global section of $K_{X_{c,n}}^{\otimes m}$, it must have non-negative vanishing at the origins of $R_c$ and $S_c$ and so 
\begin{equation}\label{alphaineq}\alpha_1\ge 0 \ \text{ and } \ 2\alpha_1+\alpha_2-m(3^c-3)\ge 0.\end{equation}
\begin{equation}\label{betaineq}\beta_2\ge 0 \ \text{ and } \ \beta_1+2\beta_2-m(3^c-3)\ge 0.\end{equation}

If $\sigma$ is of the form $x_1^{a_1}\omega_1\times x_2^{a_2}\omega_2$, then
$(\alpha_1,\alpha_2)=(a_1, 3(3^{c}-3)-2a_2)$ and $(\beta_1,\beta_2)=(m(3^{c}-3)-2a_1,a_2)$. 
Hence after simplification, Equations \eqref{alphaineq} and \eqref{betaineq} yield $a_1=a_2$. 

If $\sigma$ is the form $x_1^{a_1}\omega_1\times x_2^{a_2}y_2\omega_2$, then $(\alpha_1,\alpha_2)=(a_1, 3m(3^{c-1}-1)-2a_2-3^c)$ and $(\beta_1,\beta_2)=3m(3^{c-1}-1)-2a_1,a_2)$. After simplification, Equations \eqref{alphaineq} and \eqref{betaineq} yield $2a_2\ge 2a_2+3^c,$
which is impossible. So no such $\sigma$ can exist.

Finally, if $\sigma$ is of the form $x_1^{a_1}y_1\omega_1\times x_2^{a_2}y_2\omega_2$, then $(\alpha_1,\alpha_2)=(a_1, 3m(3^{c-1}-1)-2a_2-3^c)$ and $(\beta_1,\beta_2)= (3m(3^{c-1}-1)-2a_1-3^c,a_2)$. 
After simplification, Equations \eqref{alphaineq} and \eqref{betaineq} yield
$2a_2\ge 2a_2+2\cdot3^c,$
which is impossible. So no such $\sigma$ can exist.

Therefore the only global sections of $K_{C_g^2}^{\otimes m}$ that can correspond to global sections of $K_{X_{c,2}}^{\otimes m}$ are those of the form $x_1^a\omega_1\times x_2^a\omega_2$
for $0\le a \le m(g-1)$, which finishes the proof of the base case.  

So assume that global sections of $K_{X_{c,n-1}}^{\otimes m}$ correspond to global sections of $K_{C_g^{n-1}}^{\otimes m}$ of the form 
$x_1^a\omega_1\times \cdots \times x_{n-1}^a\omega_{n-1}$
where $0\le a\le m(g-1)$. Letting $Y_0=C_g\times X_{c,n-1}$, it follows  that a global section $\varepsilon$ of $K_{Y_0}^{\otimes m}$ corresponds to a global section of $K_{C_g^n}^{\otimes m}$ of the form 
\[\delta \times x_2^a\omega_2\times \cdots \times x_{n}^a\omega_{n},\]
where $\delta$ is a global section of $K_{C_g}^{\otimes m}$.

We have established that the form $x_{n-1}^a\omega_{n-1}$ has order of vanishing $0$ at $P_1$ and order of vanishing $m(3^{c}-3)-2a$ at $Q$.
By Lemma \ref{localaction} there exist coordinate patches $U_{n-1}$ and $V_{n-1}$ on $\mathrm{Fix}_{X_{c,n-1}}(\phi_{c,n-1})$, with local coordinates $(u_2,\ldots, u_n)$ and $(v_2,\ldots, v_n)$ respectively, on which $\phi_{c,n-1}$ acts with weights $(0,\ldots,0,g)$ and $(0,\ldots,0,1)$ respectively. Moreover the global section of $K_{X_{c,n-1}}^{\otimes m}$ corresponding to the form $x_2^a\omega_2\times \cdots \times x_{n}^a\omega_{n}$ has vanishing sequence at the point $(u_2,\ldots, u_n)=(0,\ldots,0)$ of the form 
\[(\gamma_2,\ldots, \gamma_{n-2}, a)\]
for some non-negative $\gamma_2,\ldots, \gamma_{n-1}\in \mathbb{Z}$, and vanishing sequence at the point 
$(v_2,\ldots, v_n)=(0,\ldots,0)$  of the form 
\[(\lambda_2,\ldots,\lambda_{n-2}, m(3^{c}-3)-2a)\]
for some non-negative $\lambda_2,\ldots, \lambda_{n-1}\in \mathbb{Z}$. 

It follows that $\eta_c\coloneqq \psi_g^{-1}\times \phi_{c,n-1}$ has $\mathbb{Z}/3^c\mathbb{Z}$-weights on $R_0\coloneqq U_1\times V_{n-1}$ and $S_0\coloneqq V_1\times U_{n-1}$ given by $(-1,0,\ldots,0,g)$ and $(-g,0,\ldots,1)$ respectively. 

Let $\alpha_1$ be the order of vanishing of the form $\delta$ at the point $P_1$ in $U_1$ and let $\beta_1$ be the order of vanishing of $\delta$ at the point $Q$ in $V_1$. Then, the form $\varepsilon$ has vanishing sequence at the point $(0,\ldots,0)$ in $R_0$ given by 
\[(\alpha_1,\lambda_2,\ldots,\lambda_{n-2}, m(3^{c}-3)-2a)\]
and vanishing sequence at the point $(0,\ldots,0)$ in $S_0$ given by 
\[(\beta_1,\gamma_2,\ldots, \gamma_{n-2}, a)\]

Suppose the form $\varepsilon$ corresponds to a global section of $K_{X_{c,n}}^{\otimes m}$. Then the vanishing of $\varepsilon$ on the patches $R_c$ and $S_c$ defined in Section \ref{vanishing section} must be non-negative. So by Equations \eqref{vanishing R_c} and \eqref{vanishing S_c}:
\begin{equation}\label{alpha ineq}
\frac{1}{3^c}(2\alpha_1-2a)\ge 0
\end{equation}
\begin{equation}\label{beta ineq}
\frac{1}{3^c}(\beta_1+2a-m(3^c-3))\ge 0
\end{equation}

Now by Theorem \ref{global}, the form $\delta$ is either of the form $x^b\omega$, where $0\le b\le m(g-1)$ or of the form $x^by\omega$, where $0\le b \le (m-1)(g-1)-2$. 

In the first case, namely when $\delta$ is of the form $x^b\omega$, we have $\alpha_1=a$ and $\beta_1=m(3^{c}-3)-2b.$ Hence Equations \eqref{alpha ineq} and \eqref{beta ineq} yield after simplification the condition $a=b$. 

In the second case, namely when $\delta$ is of the form $x^by\omega$, we have $\alpha_1=a$ and $\beta_1=m(3^{c}-3)-2b-3^c.$ Equations \eqref{alpha ineq} and \eqref{beta ineq} then yield the conditions $b\ge a$ and $2a-2b-3^c\ge0$, which is impossible. 

Therefore, as desired, we have shown that if $\varepsilon$ is a global section of $K_{C_g^n}^{\otimes m}$ which descends to a global section of $K_{X_{c,n}}^{\otimes m}$, then $\varepsilon$ must be of the form $x_1^a\omega_1\times x_2^a\omega_2\times \cdots x_n^a\omega_n$ for $0\le a\le m(g-1)$.
\end{proof}

\section{The Iitaka Fibration of $C_{g}^n/G$}\label{iitaka section}
Consider the Iitaka fibration of the quotient variety $C_{g}^n/G$
\[f\colon C_{g}^n/G\dashrightarrow \mathbb{P}(H^0(C_{g}^n,K_{C_{g}^n}^{\otimes m}))^G,\]
given by sending a point $x$ to its evaluation on a basis of $G$-invariant global sections of $K_{C_{g}^n}^{\otimes m}$. See \cite[Section 2.1.C]{positivity} for general facts on the Iitaka fibration of a normal projective variety. 
By Proposition \ref{koddim}, the variety $C_{g}^n/G$ and hence $X_{c,n}$ has Kodaira dimension $1$, thus the image of $f$ is a curve.  Note that after passing to a resolution $\tilde{f}\colon Z_{c,n}\rightarrow \mathbb{P}(H^0(Z_{c,n},K_{Z_{c,n}}^{\otimes m}))$ for $m$ sufficiently divisible, the smooth fibers of $\tilde{f}$ have Kodaira dimension $0$ (see \cite[Theorem 2.1.33]{positivity}).

\begin{prop}\label{fibration}
For any $c\ge 2$, $n\ge 2$, the rational map $f\colon C_{g}^n/G\dashrightarrow \mathbb{P}(H^0(C_{g}^n,K_{C_{g}^n}^{\otimes m}))^G$ has image $\mathbb{P}^1$. Moreover $f$ has reducible singular fibers above the points $0$ and $\infty$ and has singular fibers with an isolated singular point above the  $3^c$ roots of $t^{3^c}-(-1)^{n}$. 
 \end{prop}

\begin{proof}
We have the following diagram:
\begin{equation}\label{maindiagram}
\begin{tikzcd}
 C_{g}^n/G\arrow[r,dashrightarrow]{f}& \mathbb{P}(H^0( C_{g}^n,K_{ C_{g}^n}^{\otimes m}))^G\\
 C_g^n \uar[-,rightarrow]\arrow{r}{f'} &\mathbb{P}(H^0(C_g^n,K_{C_g^n}^{\otimes m}))\uar[-,dashrightarrow],
 \end{tikzcd}
\end{equation}
where the horizontal map $f'$ is the Iitaka fibration for $C_g^n$. 

Consider the composition obtained from the above diagram
\[\alpha\colon C_g^n \dashrightarrow \mathbb{P}(H^0(C_{g}^n,K_{ C_{g}^n}^{\otimes m}))^G.\]

Recall from Theorem \ref{global} that global sections of $K_{C_g}^{\otimes m}$ are of the form 
$x^a\omega$  for  $0\le a\le m(g-1).$ Observe that the only points of $C_g$ on which the form $x^a\omega$ can vanish are the points 
 $P_1\colon (x,y)=(0,1)$ and $P_2\colon (x,y)=(0,-1)$ in the coordinate patch on $C_g$ given by $\{y^2=x^{2g+1}+1\}$ and the point 
 $Q:(u,v)=(0,0)$
 in the coordinate patch given by $\{v^2=u^{2g+2}+u\}$. In fact, we have:

\begin{equation}\label{E1}
x^a\omega(P_1)=
\begin{cases}
dx^m & \mbox{if } a=0\\
0 & \mbox{otherwise},
\end{cases}
\end{equation}
\begin{equation}\label{E2}
x^a\omega(P_2)=
\begin{cases}
(-1)^mdx^m & \mbox{if } a=0\\
0 & \mbox{otherwise},
\end{cases}
\end{equation}
\begin{equation}\label{E3}
x^a\omega(Q)=
\begin{cases}
(-1)^m & \mbox{if } a=m(g-1)\\
0 & \mbox{otherwise}.
\end{cases}
\end{equation}  

Recall from the proof of Proposition \ref{koddim} that $G$-invariant global sections of $K_{ C_{g}^n}^{\otimes m}$ are of the form
\[s_a=x_1^a\omega_1\times \cdots x_n^a\omega_n,\]
for $0\le a \le m(g-1)$. Thus we may view the map $\alpha\colon C_g^n \dashrightarrow \mathbb{P}(H^0(C_{g}^n,K_{ C_{g}^n}^{\otimes m}))^G$ as the rational map sending:
\[(z_1,\ldots, z_n)\mapsto [s_0(z_1,\ldots, z_n)\colon \cdots \colon s_{m(g-1)}(z_1,\ldots, z_n)].\]

Say that $\mathbb{P}(H^0(C_{g}^n,K_{C_{g}^n}^{\otimes m}))^G$ has coordinates $[w_0\colon\cdots\colon w_{m(g-1)}]$. Then on the affine patch of 
$\mathbb{P}(H^0(C_{g}^n,K_{C_{g}^n}^{\otimes m}))^G$ given by $w_0\ne0$, the image of $\alpha$ is of the form 
$(t,t^2,\ldots, t^{m(g-1)}),$ where $t=x_1(z_1)\cdots x_n(z_n)$. The images on the other affine patches of $\mathbb{P}(H^0(C_{g}^n,K_{C_{g}^n}^{\otimes m}))^G$ take similar forms. Hence the image of $\alpha\colon C_g^n \dashrightarrow \mathbb{P}(H^0(C_{g}^n,K_{C_{g}^n}^{\otimes m}))^G$ is the rational curve $\mathbb{P}^1$ from which it follows that the image of $f\colon C_{g}^n\rightarrow \mathbb{P}(H^0(C_{g}^n,K_{C_{g}^n}^{\otimes m}))^G$ is $\mathbb{P}^1$ as well.

Consider the codimension $1$ subvarieties of $C_g^n$ of the form $C_g^{n-1} \times P_{i}$ for $i\in \{1,2\}$ up to permutation of factors. Observe from \eqref{E1}, \eqref{E2}, and \eqref{E3} that $\alpha$ sends the open subset $(C_g-Q)^{n-1}\times P_i$ of such a subvariety to the point $[1\colon0\colon\cdots\colon0]$ in $\mathbb{P}(H^0(C_g^n,K_{C_g^n}^{\otimes m}))^G$, which corresponds to the point $[1\colon 0]$ in $\mathbb{P}^1$.  Moreover for the different permutations of the position of the $P_i$, the action of the group $G$ on $C_g^n$ does not identify these various open subvarieties. 
Namely, the fiber of $f$ above the point $[1\colon 0]$ contains all $2n$ of these open subvarieties.  In particular, the fiber of $[1\colon 0]$ is singular and reducible. 

Similarly, consider the $n$ subvarieties of $C_g^n$ of the form $C_g^{n-1}\times Q$. Then $\alpha$ sends the open subsets $(C_g-\{P_1,P_2\})^{n-1}\times Q$ of these subvarieties to the point $[0\colon 1]$ in $\mathbb{P}^1$ and since the action of the group $G$ does not identify these open subsets, they all lie in the fiber of $[0\colon 1]$ under $f$. In particular, the fiber of $[0\colon1]$ is singular and reducible.

Now let us consider the fibers of $\alpha$ away from the points $[1\colon 0]$ and $[0\colon 1]$ in $\mathbb{P}^1$. Note that away from these two points, the image of $\alpha$ is given by points $(t,t^2,\ldots, t^{m(g-1)}),$ where $t=x_1(z_1)\cdots x_n(z_n)$ is not equal to zero. Hence 
the fibers of $\alpha$ away from $[1\colon 0]$ and $[0\colon 1]$ are then the $G$-invariant hypersurfaces $F_t$ in $C_g^n$ given by $x_1\cdots x_n=t$. In other words, these fibers are 
defined by the affine equations in $(\mathbb{A}^2)^n$:
\begin{equation}\label{F equation}(y_1^2=x_1^{2g+1}+1,\  y_2^2=x_2^{2g+1}+1,\ \ldots,\  y_n^2=x_n^{2g+1}+1, \ x_1\cdots x_n=t),\end{equation}
where we are assuming $t\ne0$. 

On the affine patch we have described, such fibers $F_t$ have Jacobian:
\begin{equation*}
\left(\begin{array}{ccccccc}
(2g+1)x_1^{2g} &2y_1 & 0&0&\cdots &0&0\\
0&0&(2g+1)x_2^{2g} &2y_2&\cdots &0&0\\
0&0&0&0&\cdots &(2g+1)x_n^{2g} &2y_n \\
x_2\cdots x_n&0&x_1x_3\cdots x_n&0&\cdots &x_1\cdots x_{n-1}&0
\end{array}\right)
\end{equation*}

Hence the fiber $F_t$ is singular if $y_1=\cdots=y_n=0$. When this is the case, then for each $i=1,\ldots, n$ we have that $x_i$ satisfies the equation $x_i^{2g+1}+1=0$, namely $x_i$ is of the form $\xi^{2\gamma_i+1}$, where $\xi$ is a primitive $2\cdot 3^c$-th root of unity and $0\le \gamma_i\le 3^c-1$. Namely we have 
\[t=\xi^{\left(2\left(\sum_{i=1}^n \gamma_i\right)+m\right)}\]
and so 
\[t^{3^c}=\xi^{n\cdot 3^c}=(-1)^n.\]

In other words, it $t\in \mathbb{C}^*$ is such that $t^{3^c}=(-1)^n$, then the fiber $F_t$ is singular and has singularities at the points of the form
\[((x_1,y_1),\ldots, (x_n,y_n))=((\xi^{2\gamma_1+1},0),\ldots, (\xi^{2\gamma_n+1},0)).\]

Note that since $\zeta$ is an even power of $\xi$, the action of the group $G$ permutes these singular points on $F_t$. Namely, in the image of $F_t$ in $C_g^n/G$ these points are all identified to a single singular point. 

Note that if we consider some other affine patch of the fiber $F_t$, we substitute in \eqref{F equation} equations of the form $v_i^2=u_i^{2g+2}+u_i$ and then the equation $x_1\cdots x_n=t$ becomes $\frac{x_{i_1}\cdots x_{i_k}}{u_{j_1}\cdots u_{j_\ell}}=t$. Considering the Jacobian as above yields that the fiber $F_t$ is singular if $y_{i_1}=\cdots=y_{i_k}=v_{j_1}=\cdots=v_{j_\ell}=0$. Since $\frac{x_{i_1}\cdots x_{i_k}}{u_{j_1}\cdots u_{j_\ell}}=t$, we know that none of the $u_{j_r}$ can be equal to zero. Hence the points described by $y_{i_1}=\cdots=y_{i_k}=v_{j_1}=\cdots=v_{j_\ell}=0$ are in fact the same points as those described by just $y_1=\cdots=y_n=0$. Since we have already shown that these points are identified to a single singular point in the image of $F_t$ in $C_g^n/G$, it follows that the image of $F_t$ in $C_g^n/G$ has a single singular point. 
\end{proof}

%%%%%%%%%%%%%%%%%%%%%%%%%%%%%%%%%%%%%%%%%%%%%%%%%%%%%%%%%%%%
%%%%%%%%%%%%%%%%%%%%%%%%%%%%%%%%%%%%%%%%%%%%%%%%%%%%%%%%%%%%%
\section{The elliptic surface case}\label{surface section}
We now focus our attention on the case when $n=2$. In this case, for $m$ sufficiently divisible, the rational map $f\colon C_{g}^2/G\dashrightarrow \mathbb{P}^1\subset \mathbb{P}(H^0(C_{g}^2,K_{C_{g}^2}^{\otimes m}))^G$ studied in Proposition \ref{fibration}, can be resolved to yield an elliptic fibration $\tilde{f}\colon Z_{c,2}\rightarrow \mathbb{P}^1$. 
 To better understand this minimal elliptic surface $Z_{c,2}$, we will make use of the fact that minimal resolutions of cyclic quotient singularities are well-understood in dimension $2$. 

The action of $\eta_c=\psi_g^{-1}\times \psi_g$ on the product $C_g\times C_g$ has $9$ fixed points: five of the form $(P_i,P_j)$ or $(Q,Q)$, which we refer to as \textit{Type I fixed points}, and four of the form $(P_i,Q)$ or $(Q,P_i)$, which we refer to as \textit{Type II fixed points}.  We will also refer to these points as the \textit{Type I} and \textit{Type II singular points}, respectively, of the quotient $C_g^2/G=(C_g\times C_g)/\langle \eta_c\rangle$. 

Observe that $\eta_c$ acts around Type I fixed points with $\mathbb{Z}/3^c\mathbb{Z}$-weights $(-1,1)$ in the $(P_i,P_j)$ case and $\mathbb{Z}/3^c\mathbb{Z}$-weights $(-g,g)$ in the $(Q,Q)$ case. Similarly $\eta_c$ acts around Type II fixed points with $\mathbb{Z}/3^c\mathbb{Z}$-weights $(-1,g)$ in the $(P_i,Q)$ case and $\mathbb{Z}/3^c\mathbb{Z}$-weights $(-g,1)$ in the $(Q,P_i)$ case. In particular then, note that the Type II singular points are non-canonical singularities. 

%%%%%%%%%%%%%%%%%%%%%%%%%%%%%%%%%%%%%%%%%%%%%%%%%%%%
%%%%%%%%%%%%%%%%%%%%%%%%%%%%%%%%%%%%%%%%%%%%%%%%%%%%

\subsection{Resolving the singular points of $C_g^2/G$}\label{hirzebruchjung section}

To understand the resolutions of the singular points of $C_g^2/G$, we make use of established facts about surface cyclic quotient singularities and Hirzebruch-Jung resolutions.  A brief survey of these can be found in \cite[Section 2.4]{kollar} and more detailed explanations can be found in \cite{reid} and \cite{BHPV}.

Suppose the cyclic group $\mathbb{Z}/r\mathbb{Z}$ acts on $\mathbb{C}^2$ via
$(z_1,z_2) \mapsto (\epsilon z_1, \epsilon^a z_2),$
for some $a$ coprime to $r$, where $\epsilon$ is a primitive $r$-th root of unity. Then the minimal resolution of the singularity at $(0,0)$ in the quotient $\mathbb{C}^2/\mathbb{Z}/r\mathbb{Z}$ is encoded by the continued fraction expansion:
\[\frac{r}{a}=b_0-\frac{1}{b_1-\frac{1}{b_2-\frac{1}{b_3-\frac{1}{\cdots}}}}.\]
More precisely, the minimal resolution of this singularity is a chain of $s+1$ exceptional curves $E_0, E_1,\ldots, E_s$ with nonzero intersection numbers $E_i.E_i=-b_i$ and $E_i.E_{i+1}=1$ \cite[Proposition 2.32]{kollar}. The sequence $[b_0,b_1,b_2,b_3,\ldots b_s]$ is called the \emph{Hirzebruch-Jung expansion} of the singularity. 

Therefore since $\eta_c$ acts around Type I fixed points with $\mathbb{Z}/3^c\mathbb{Z}$-weights $(-1,1)$  and $(-g,g)$, the Type I fixed points have Hirzebruch-Jung expansion
\[\underbrace{[2,\ldots ,2]}_{(3^c-1)-\mathrm{times}}.\]

Thus the Type I singular points of $C_g^2/G$ are DuVal singularities of type $A_{3^c-1}$ whose minimal resolutions consist of a chain of $3^c-1$ rational curves, each with self-intersection $-2$.

Similarly, since the Type II fixed points are acted on by $\eta_c$ with weights $(-1,g)$ and $(-g,1)$, the Type II singular points  of $C_g^2/G$ have Hirzebruch-Jung expansion
$[2,g+1].$ Hence the minimal resolution of each Type II singular point consists of a chain of two rational curves, one denoted $T$ with self-intersection $-2$ and one denoted $S$ with self-intersection $-(g+1)$.

%%%%%%%%%%%%%%%%%%%%%%%%%%%%%%%%%%%%%%%%%%%%%%%%%%%%
%%%%%%%%%%%%%%%%%%%%%%%%%%%%%%%%%%%%%%%%%%%%%%%%%%%%
\subsection{Weights and vanishing on the curves $T$ and $S$}\label{S and T phi}

Since the Type II singular points of $C_g^2/G$ are not canonical singularities, they will be of special interest to us in understanding both the geometry of the surface $Z_{c,2}$ and of the threefold $Z_{c,3}$. In this section, we thus pay special attention for use later to the local action of the automorphisms $\eta_c$ and $\phi_{c,2}$ on the images of Type II singular points in $X_{c,2}$ and $Z_{c,2}$.

Without loss of generality, let us consider a Type II point  on $C_g\times C_g$ of the form $P_i\times Q$ for some $i\in \{1,2\}$. As discussed in Sections \ref{phi action section}, this point is covered by a coordinate patch $R_0$ with local coordinates $(z_{0,1},z_{0,2})$ on which the automorphism $\eta_c=\psi_g^{-1}\times \psi_g$ acts with $\mathbb{Z}/3^c\mathbb{Z}$-weights $(-1,g)$ and the automorphism $\mathrm{id}\times \psi_g$ acts with $\mathbb{Z}/3^c\mathbb{Z}$-weights $(0,g)$.

After performing a sequence of two blow-ups along the fixed locus of $\eta_c$, we have that $\eta_c$ acts on the resulting coordinate patches with $\mathbb{Z}/3^c\mathbb{Z}$-weights
\[(-1, g+2)\ \ (-(g+2),g+1)\ \ (-(g+1),0)\ \ (0,g)\]
and hence that $\mathrm{id}\times \psi_g$ acts with corresponding $\mathbb{Z}/3^c\mathbb{Z}$-weights
\[(0,g)\ \ (-g,g)\ \ (-g,2g)\ \ (-2g,g).\]

Note that coordinate patch on which $\eta_c$ acts with weights $(0,g)$ is exactly the coordinate patch $R_0''$ introduced in Section \ref{local weight section}. We showed in that section that the image $R_c$ of $R_0''$ in $Y_c$ is obtained by taking a sequence of $c-1$ quotients by $\mathbb{Z}/3\mathbb{Z}$ since all subsequent blow-up maps will be isomorphisms on this patch. Hence the automorphism $\eta_c$ acts on $R_c$ with weights $(0,g)$ and the automorphism $\mathrm{id}\times \psi_g$ acts on $R_c$ with weights $(-2g,g)=(1,g)$.  Moreover, since a global section $s_a=x_1^a\omega_1\times x_2^a\omega_2$ of $K_{C_g^2}^{\otimes m}$ has vanishing sequence on $P_i\times Q$  given by $(a,m(3^c-3)-2a)$, the form $s_a$ has vanishing sequence at the origin of $R_c$ given by $(a, 0)$.

Now consider the coordinate patch $T_0''$ of $Y_0''$ on which $\eta_c$ acts with weights $(-1,g+2)=(-1,\frac{3^c+3}{2})$. Let $T_1$ denote the image of this patch after taking the $\mathbb{Z}/3\mathbb{Z}$-quotient needed to pass from $Y_0''$ to $Y_1$. Then $\eta_c$ acts on $T_1$ with weights $\left(-1, \frac{3^{c-1}+1}{2}\right)$.  Suppose that the sequence of two blow-ups needed to obtain $Y_0''$ from $Y_0$ introduced new local coordinates $[z_{0,1}': z_{0,2}']$ followed by $[z_{0,1}'': z_{0,2}'']$ such that the $T_0''$ patch is given by $z_{0,1}'\ne 0$ followed by $z_{0,1}''\ne 0$. 

Inductively, if the sequence of blowups $Y_i''\rightarrow Y_i'\rightarrow Y_i$ introduces new local coordinates $[z_{i,1}':z_{i,2}']$ and $[z_{i,1}'':z_{i,2}'']$, define $T_i$ to be the image in $Y_i$ of the $z_{i-1,1}'\ne 0$ patch followed by the  $z_{i-1,1}''\ne 0$ patch. Suppose $\eta_c$ acts with weights $\left(-1, \frac{3^{c-i}+1}{2}\right)$ on $T_i$. Then after one blow-up $\eta_c$ acts with weights $\left(-1, \frac{3^{c-i}+3}{2}\right)$ on the $z_{i,1}'\ne0$ patch. Since $-\frac{3^{c-i}+3}{2}$ is divisible by $3$, the blow-up $Y_i''\rightarrow Y_i'$ is an isomorphism on the $z_{i,1}'\ne0$ patch, and thus $\eta_c$ indeed acts with weights $\left(-1, \frac{3^{c-(i+1)}+1}{2}\right)$ on $T_{i+1}$. 

Hence by induction, we have that indeed $\eta_c$ acts with weights $\left(-1, \frac{3^{c-i}+1}{2}\right)$ on $T_i$ for each $1\le i\le c$ and that to pass from $T_i$ to $T_{i+1}$ we perform one non-trivial blow-up followed by one $\mathbb{Z}/3\mathbb{Z}$ quotient. In particular, the automorphism $\eta_c$ acts with weights $(-1, 1)$ on the patch $T_c$ in $Y_c=X_{c,2}$ and the automorphism $\mathrm{id}\times \psi_g$ acts with weights $(0,g)$ on $T_c$. 

Using the inductive construction of the coordinate patch $T_c$ together with Equations \eqref{quotienteq} and \eqref{blowupeq}, we deduce that a global section of $K_{C_g^2}^{\otimes m}$ with vanishing sequence $(\alpha_1, \alpha_2)$ on $P_i\times Q$ has vanishing sequence on the origin of $T_c$ given by
\[\left
(
\frac{1}{3}\big( \cdots \big( \frac{1}{3}\big(\frac{1}{3}(\alpha_1+2\alpha_2)+\alpha_2-m\big)\cdots+\alpha_2-m\big),\alpha_2\right)=\left(\frac{1}{3^c}\big(\alpha_1+\frac{3^c+1}{2}\alpha_2-m\frac{3^c-3}{2}\big),\alpha_2\right).\]
In particular, the form $s_a$ has vanishing sequence on the origin of $T_c$ given by 
\[\left(\frac{1}{2}\big(m(3^c-3)-2a\big),m(3^c-3)-2a\right).\]

Now by construction of $Y_c=X_{c,2}$, the image of the point  $P_i\times Q$  in $X_{c,2}$ consists of a chain of rational curves, one extreme end of which is covered by the patch $R_c$ and the other extreme end of which is covered by the patch $T_c$. Passing from $X_{c,2}$ to the minimal surface $Z_{c,2}$ involves contracting all but the two extremal curves in the chain, which become the curves $T$ and $S$ in $Z_{c,2}$. 

This chain of $2$ curves, the curve $T$ followed by the curve $S$, is then covered by three coordinate patches: the patch $T_c$, a new patch $W_c$, and then the patch $R_c$. Hence our calculations yield that the automorphism $\eta_c$ acts respectively on these patches with local $\mathbb{Z}/3^c\mathbb{Z}$ weights 
$(-1,1)$, $(-1, 0)$, and $(0,g)$
and that the automorphism  $\phi_{c,2}=\mathrm{id}\times \psi_g$ acts respectively with weights 
\begin{equation}\label{phi weights}
(0,g),\ (-g, -1),\   \text{ and }(1,g).
\end{equation} 
Moreover, our calculations together with  Equation \eqref{blowupeq} yield that the form $s_a$ has vanishing sequence at the origins of the three coordinate patches $T_c,$ $W_c$, and $R_c$ given respectively by
\begin{equation}\label{curvevanishing}
\left(\frac{1}{2}\big(m(3^c-3)-2a\big),m(3^c-3)-2a\right),\  \left(0,\frac{1}{2}\big(m(3^c-3)-2a\big)\right),\   \text{ and }(a,0).
\end{equation}

\subsection{The canonical bundle $K_{Z_{c,2}}$}
From the above calculations we observe the following about the canonical bundle $K_{Z_{c,2}}$ of the minimal surface $Z_{c,2}$. 
\begin{prop}\label{canonicalbundle} For $c\ge 2$, the canonical bundle $K_{Z_{c,2}}$ is basepoint free. 
\end{prop}
\begin{proof}
Consider the rational map $h\colon Z_{c,2}\dashrightarrow \mathbb{P}(H^0(Z_{c,2},K_{Z_{c,2}}))$ induced by the canonical bundle $K_{Z_{c,2}}$. Observe that $h$ fits into a diagram
\begin{equation}
\begin{tikzcd}
Z_{c,2}\rar[-,dashrightarrow]{h}&\mathbb{P}(H^0(Z_{c,2},K_{Z_{c,2}}))\\
C_g^2\uar[-,dashrightarrow]\rar[-,dashrightarrow] &\mathbb{P}(H^0(C_g^2,K_{C_g^2}))\uar[-,dashrightarrow],
\end{tikzcd}
\end{equation}
where the horizontal maps are given by evaluation on forms $s_a\coloneqq x_1^a\omega_1\times x_2^a\omega_2$ for $0\le a\le g-1$ and the rational vertical map on the left is the sequence of blow-ups, blow-downs, and quotients needed to obtain $X_{c,2}$ from $C_g^2$ followed by the birational map $ X_{c,2}\rightarrow Z_{c,2}$. 

Observe that the points of $C_g^2$ on which all the $s_a$ vanish are exactly the Type II fixed points. Thus to prove that $K_{Z_{c,2}}$ is basepoint free, we just need to ensure that not all of the $s_a$ vanish on the image in $Z_{c,2}$ of a Type II fixed point. But the image in  $Z_{c,2}$ of a Type II fixed point is the pair of curves $T$ and $S$ covered by the three coordinate patches $T_c,$ $W_c$, and $R_c$ defined in Section \ref{S and T phi}. So the result follows from Equation \eqref{curvevanishing}.
\end{proof}

\subsection{The elliptic fibration $h\colon Z_{c,2}\rightarrow \mathbb{P}^1$}
It follows from Proposition \ref{canonicalbundle} that the Iitaka fibration $\tilde{f}\colon Z_{c,2}\rightarrow \mathbb{P}^1\subset  \mathbb{P}(H^0(Z_{c,2},K_{Z_{c,2}}^{\otimes m}))$ obtained by resolving the rational map $f\colon C_g^2/G\dashrightarrow \mathbb{P}^1\subset  \mathbb{P}(H^0(C_g^2,K_{C_g^2}^{\otimes m}))^G$ studied in Proposition \ref{fibration} may be obtained by letting $m=1$. In this case, we just obtain the map
$h\colon Z_{c,2}\rightarrow \mathbb{P}(H^0(Z_{c,2},K_{Z_{c,2}}))$ from Proposition \ref{canonicalbundle}. In Proposition \ref{sing} below, we study in detail the geometry of this elliptic fibration, which we illustrate in Figure \ref{ellipticfigure}. 

We make use of Kodaira's classification, in \cite{kodaira} and \cite{kodaira2}, of the possible singular fibers of an elliptic surface. For a survey of the possible fiber types, see \cite[I.4]{miranda} and \cite[Section 4]{elsur}.

As we will see, the two kinds of singular fibers that appear in the fibration $h$ are singular fibers of type $I_b$ for $b>0$ and singular fibers of type $I_b^*$ for $b\ge 0$. Singular fibers of type $I_b$ consist of $b$ smooth rational curves meeting in a cycle, namely meeting with dual graph the affine Dynkin diagram $\tilde{A}_b$. Singular fibers of type $I_b^*$ consist of $b+5$ smooth rational curves meeting with dual graph the affine Dynkin diagram $\tilde{D}_{b+4}$.

Recall that each Type II singularity in the quotient $C_g^2/G$ yields two rational curves $T$ and $S$ in $Z_{c,2}$, where $T$ has self-intersection $-2$ and $S$ has self-intersection $-(g+1)$. Let $\delta_1=Q\times P_1$, $\delta_2=Q\times P_2$, $\delta_3=P_1\times Q$, and $\delta_4=P_2\times Q$ denote these four Type II singular points and let $S_1$, $S_2$, $S_3$, and $S_4$ denote each of their respective $-(g+1)$-curves in $Z_{c,2}$ and $T_1$, $T_2$, $T_3$, and $T_4$ their respective $(-2)$-curves in $Z_{c,2}$. 

In the dimension $2$ case, Proposition \ref{fibration} then yields:

\begin{prop}\label{sing}For $c\ge 2$, the elliptic surface $h\colon Z_{c,2}\rightarrow \mathbb{P}^1$ has $3^c+2$ singular fibers: one of type $I_{4\cdot3^c}$ located at $0$, one of type $I_{3^c}^*$ located at $\infty$, and the remaining $3^c$ of type $I_1$ and located at the points $\zeta^i$, for $\zeta$ a primitive $3^c$-th root of unity.  Additionally, each of the rational curves $S_1$, $S_2$, $S_3$, and $S_4$ coming from the resolution of a Type II singular point corresponds to a section of $\tilde{f}$. 
\end{prop}

\begin{figure}\label{ellipticfigure}[h]
\centering{
\includegraphics
[width=.8\textwidth]
{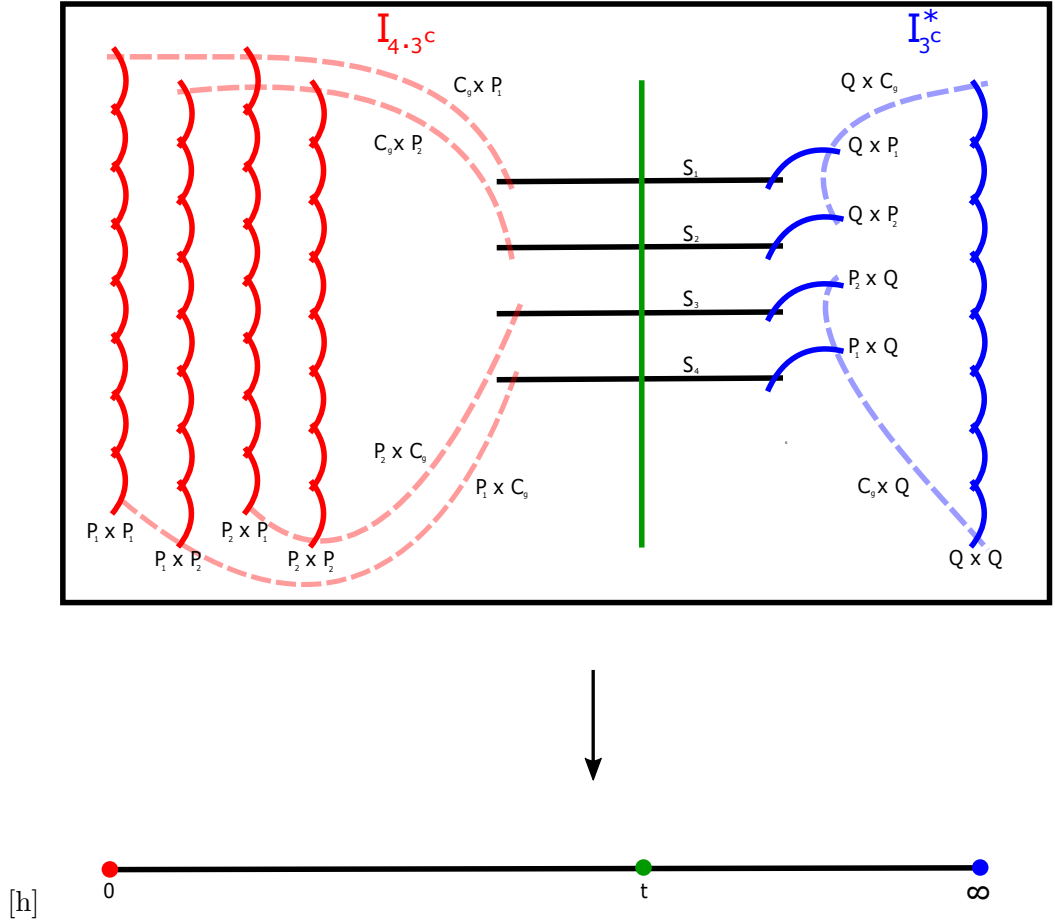}
\caption{The elliptic surface $h\colon Z_{c,2}\rightarrow \mathbb{P}^1$}}
\end{figure}

\begin{proof}
Consider the diagram
\begin{equation}
\begin{tikzcd}\label{maindiagram2}
Z_{c,2}\rar[-,rightarrow]{h}&\mathbb{P}^1\subset\mathbb{P}(H^0(Z_{c,2},K_{Z_{c,2}}))\\
C_g^2/G\uar[-,dashrightarrow]\rar[-,dashrightarrow]{f} &\mathbb{P}^1\subset\mathbb{P}(H^0(C_g^2,K_{C_g^2}))^G\uar[-,rightarrow]\\
C_g^2\uar[-,rightarrow]\rar[-,dashrightarrow] &\mathbb{P}(H^0(C_g^2,K_{C_g^2}))\uar[-,rightarrow],
\end{tikzcd}
\end{equation}
where $f$ is the rational map studied in Proposition \ref{fibration} and the rational map $C_g^2/G\dashrightarrow Z_{c,2}$ on the top left of the diagram is obtained by resolution of singularities. 

By Proposition \ref{fibration}, the rational map $f$  has singular fibers at $0$, $\infty$, and $\zeta^i$ in $\mathbb{P}^1$. Thus we consider the fibers of $h$ above these points. Let us begin by focusing on the fibers of $h$ above the points $0$ and $\infty$. 

Consider the rational map $\alpha\colon C_g^2 \dashrightarrow \mathbb{P}^1\subset \mathbb{P}(H^0(C_g^2, K_{C_g^2})^G$ from the proof of Proposition \ref{fibration}. Recall that $\alpha$ may be viewed as the rational map given by $(z_1,z_2)\mapsto [s_0(z_1,z_2): \cdots : s_{g-1}(z_1,z_2)]$, where $s_a\coloneqq x_1^a\omega_1\times x_2^a\omega_2$ for $0\le a\le g-1$. We know that the points of $C_g^2$ on which all the $s_a$ vanish are exactly the Type II fixed points. 

For any set of points $A$ on the curve $C_g$, let $C_g-A$ denote the complement in $C_g$ of this set of points. Then from what we have established, we know that the preimage under $\alpha$ of the point $[1:0:\cdots:0]$ consists of the union of open curves
\begin{equation}\label{preimage1}
(P_1\times (C_g-\{Q\}))\cup (P_2\times (C_g-\{Q\}))\cup ((C_g-\{Q\})\times P_1)\cup ((C_g-\{Q\})\times P_2).
\end{equation}
Similarly, the preimage under $\alpha$ of the point $[0:\cdots:0:1]$ contains the union of open curves
\begin{equation}\label{preimage2}
(Q\times Q)\cup (Q\times (C_g-\{P_1,P_2\}))\cup ((C_g-\{P_1,P_2\})\times Q).
\end{equation}

In particular, the image under $\alpha$ of the fixed points in $C_g^2$ of the form $(P_i,P_j)$ is the point $0$ in  $\mathbb{P}^1$. Each such point has image in $Z_{c,2}$ consisting of a chain of $3^c-1$ rational curves and so by the Diagram (\ref{maindiagram2}), the fibration $h\colon Z_{c,2}\rightarrow \mathbb{P}^1$ must send all of these $3^c-1$ rational curves to the point $0$.

Moreover, using Diagram \eqref{maindiagram2} in conjunction with Equations \eqref{preimage1}, since $h$ is a morphism we must have that the strict transforms in $Z_{c,2}$ of the curves $P_1\times C_g,$ $P_2\times C_g,$  $C_g\times P_1,$  and $C_g\times P_2$
also get sent to $0$. Note that the strict transform of $C_g\times P_j$ will intersect the chain of rational curves resolving the singularity $P_i\times P_j$ at one end of the chain and the strict transform of $P_i\times C_g$ will intersect the chain at the other end of the chain. 

Similarly, by Equation \eqref{preimage2} the image of the fixed point $(Q,Q)$ in $C_g^2$ will be sent by $\alpha$ to the point $\infty$ in  $\mathbb{P}^1$. Since such a fixed point has image in $Z_{c,2}$ consisting of a chain of $3^c-1$ rational curves, Diagram \eqref{maindiagram2} yields that $f$ must send all of these rational curves to the point $\infty$.

Moreover, by Diagram \eqref{maindiagram2} in conjunction with \eqref{preimage2}, since $h$ is a morphism we must have that the strict transforms in $Z_{c,2}$ of the curves 
$Q\times C_g $ and $C_g\times Q$
 get sent to $\infty$ as well. Again, the strict transform of $Q\times C_g$ will intersect the chain of rational curves resolving the singularity  $(Q,Q)$ at one end and the strict transform of $C_g\times Q$ will intersect the chain at the other end. 
 
Therefore we have established that $h$ sends to the point $0$ in $\mathbb{P}^1$ the strict transforms in $Z_{c,2}$ of the curves
 \begin{equation}\label{Pcurves}P_1\times C_g,\  P_2\times C_g, \ C_g\times P_1, \text{ and } C_g\times P_2\end{equation}
 and sends to the point $\infty$ the strict transforms of the curves 
\begin{equation}\label{Qcurves}Q\times C_g \text{ and } C_g\times Q.\end{equation}

Each of the four Type II fixed points $\delta_1,$ $\delta_2$, $\delta_3,$ and $\delta_4$ in $C_g^2$ has one of the curves in (\ref{Pcurves}) and one of the curves in (\ref{Qcurves}) passing through it. Note for instance that if $\delta_j$ is of the form $Q\times P_i$, then its resolution in $Z_{c,2}$ intersects the curve $Q\times C_g$ at the end of the curve $T_j$ away from $S_j$ and intersects the curve $C_g\times P_i$ at the end of $S_j$ away from $T_j$.

Moreover, by the adjunction formula, the curves on $Z_{c,2}$ contracted by the map $h$ are exactly those with self-intersection $-2$, meaning that $h$ contracts the curves $T_j$ and maps the curve $S_j$ to all of $\mathbb{P}^1$ for each $1\le j\le 4$.  Since each curve $S_j$ intersects either the curve $Q\times C_g$ or the curve $ C_g\times Q$, both of which get sent to the point $\infty$ by $\alpha$, it follows that $f$ sends all the $S_j$ curves to$\infty$ as well. 

In summary, the fiber of $h$ above the point $0$ in $\mathbb{P}^1$ is a cycle consisting of the four sets of $3^c-1$ rational curves coming from the resolutions of the points $P_i\times P_j$ together with the four curves in (\ref{Pcurves}). Hence the fiber consists of $4(3^c-1)+4=4\cdot 3^c$
rational curves and thus is a fiber of type $I_{4\cdot 3^c}$ in Kodaira's classification. Similarly, the fiber above the point $\infty$ 
 consists of the $3^c-1$ rational curves resolving the singularity $Q\times Q$ together with the curves in (\ref{Qcurves}) and the four $(-2)$-curves $T_1$, $T_2$, $T_3$, $T_4$. Hence, the fiber consists of a chain of $3^c+1$ rational curves, where each curve on the ends of the chain has two additional curves coming off it.  This is a fiber of type $I_{3^c}^*$. See figure \ref{ellipticfigure} for a pictorial representation of this arrangement. 
 
It remains to identify the fibers of $h$ occurring above the points $\zeta^i$. Note that by the proof of Proposition \ref{fibration} the fibers of $f$ above the $\zeta^i$ have a singularity at the single point on the fiber which is the image in $C_g^2/G$ of the points in $C_g^2$ of the form $((x_1, 0), (x_2, 0))$, where the $x_j$ are of the form $\xi^{2\gamma_j+1}$ for $\xi$ a primitive $2\cdot 3^c$-th root of unity and $0\le \gamma_j\le 3^c-1$. In particular, $x_j\ne \pm 1$ and so while the corresponding point in  $C_g^2/G$ is singularity of the fiber of $f$, it is not a singular point of the surface  $C_g^2/G$. In particular, it remains a singular point of the fiber of $h$ above $\zeta^i$. So $h$ has fibers with an isolated singular point above the points $\zeta^i$ in $\mathbb{P}^1$. To determine these fibers more precisely, note that from \cite[Proposition 5.16]{cossec},
for a complex elliptic surface $\varphi:S\rightarrow C$ with fiber $F_v$ at $v\in C$ having $m_v$ components, we have 
\begin{equation}\label{euler characteristic formula}\chi_{\mathrm{top}}(S)=\sum_{v\in C}e(F_v),\end{equation}
 where $e(F_v)$ is $0$ if $F_v$ is smooth, is $m_v$ if $F_v$ is of type $I_n$, and is $m_v+1$ otherwise. Since the surface $Z_{c,2}$ has $q=0$ and geometric genus $p_g=g$, its geometric Euler number is $g+1$, and so by Noether's formula $\chi_{\mathrm{top}}(Z_{c,2})=12(g+1)=6\cdot3^c+6.$ Moreover, the identified fibers of $h$ of type $I_{4\cdot 3^c}$ and $I_{3^c}^*$ above $0$ and $\infty$ respectively contribute $4\cdot3^c+((3^c+5)+1)=5\cdot 3^c+6$ to the right hand side of Equation \eqref{euler characteristic formula}. Hence the $3^c$ singular fibers of $h$ with an isolated singularity contribute exactly $3^c$ to the right hand side of Equation \eqref{euler characteristic formula}. It follows that $m_v=1$ for each of these singular fibers, meaning that each such fiber is of type $I_1$ in Kodaira's classification, as claimed. This completes the analysis of the singular fibers of $h$.

It thus only remains to verify that each of the curves $S_j$ for $1\le j\le 4$ corresponds to a section of $h$. Since we know that $h$ maps $S_j$ surjectively onto $\mathbb{P}^1$, it just remains to verify that for each $t\in \mathbb{P}^1$ there is a unique $s\in S_j$ such that $h(s)=t$. 

Without loss of generality suppose that the point $\delta_j$ is of the form $Q\times P_i$,  as a symmetric argument will work for points of the form $P_i\times Q$. Let $t\in \mathbb{P}^1$ and consider the fiber $\mathcal{F}_t=h^{-1}(t)$. Since we have already determined the points of intersection of $\mathcal{F}_0$ and $\mathcal{F}_\infty$ with $S_j$ we may assume $t\ne 0,\infty$. 

Now $\mathcal{F}_t$ is the image in $Z_{c,2}$ of the curve $F_t$ in $C_g^2$ given by the equation $x_1x_2=t$.  Recall that $Q\times P_i$ is given in local coordinates by $(v_1,x_2)=(0,0)$ and so near $\delta_j$ the curve $F_t$ is given by $u_1^{-1}x_2=t$.  We may rewrite this as $\gamma(v_1)^{-1}x_2=t$, for $\gamma(v_1)$ a continuous function of degree $2$ in $v_1$. Hence, close to $\delta_j$, the curve $F_t$ has coordinates given by $(v_1,t\gamma(v_1))$.  Thus the slope of $F_t$ at $\delta_j$ is 
$\lim_{v_1\to 0}\frac{t\gamma(v_1)}{v_1}=0.$

It follows that the strict transform $F_t'$ of $F_t$ in the blow-up of $C_g^2$ at $\delta_j$ intersects the exceptional curve $E_0'$ with coordinates $[z_{0,1}':z_{0,2}']$ at the point $[z_{0,1}':z_{0,2}']=[1:0]$. Taking the coordinate patch $z_{0,1}'\ne0$ yields local coordinates $(v_1,z_{0,2}')$ and $F_t'$ intersects $E_0'$ at the point $(v_1,z_{0,2}')=(0,0)$. Moreover, near this point, the curve $F_t'$ has coordinates $(v_1,z_{0,2}')=\left(v_1, \frac{t\gamma(v_1)}{v_1}\right)$, since $v_1z_{0,2}'=z_{0,1}'x_2$. Hence the slope of $F_t'$ at the point $(v_1,z_{0,2}')=(0,0)$ is 
$\lim_{v_1\to 0}\frac{t\gamma(v_1)}{v_1^2}=t.$

Since the point $(v_1,z_{0,2}')=(0,0)$ gets blown up in the transformation $Y_0''\rightarrow Y_0$, the strict transform $F_t''$ of the curve $F_t'$ after this blowup intersects the exceptional curve at the point with coordinate $t$. Moreover, observe that this point with coordinate $t$ is covered by the coordinate patch $R_0''$ introduced in Section \ref{local weight section} and discussed in more detail in the surface case in Section \ref{S and T phi}. Since the coordinate patch $R_c$ in $Y_c=X_{c,2}$ is obtained from $R_0''$ by a sequence of $c-1$ quotients by $\mathbb{Z}/3\mathbb{Z}$, this point corresponds to the point with coordinate $t^{3^{c-1}}$ on the image of this exceptional curve in $X_{c,2}$. But since this exceptional curve does not get contracted in passing from $X_{c,2}$ to $Z_{c,2}$ (see Section \ref{S and T phi}), this intersection point is also the point with coordinate $t^{3^{c-1}}$ on the image of this exceptional curve in $Z_{c,2}$, which is just the curve $S_j$. It follows that the curve $\mathcal{F}_t$ intersects $S_j$ at the point of $S_j$ with coordinate $t^{3^{c-1}}$. Hence $s=t^{3^{c-1}}$ is the unique point in $S_j$ such that $h(s)=t$ and therefore $S_j$ indeed corresponds to a section of $h$.

\end{proof}

%%%%%%%%%%%%%%%%%%%%%%%%%%%%%%%%%%%%%%%%%%%%%%%

The Mordell-Weil group of an elliptic fibration $\varphi:S\rightarrow C$ is the group of $K$-rational points on the generic fiber of $\varphi$, where $K=\mathbb{C}(C)$. Such an elliptic surface $S$ is called \emph{extremal} if it has maximal Picard rank $\rho(S)$, meaning $\rho(S)=h^{1,1}(S)$, and its Mordell-Weil group has rank $r=0$. 

Schreieder proves in \cite[Section 8.2]{schreieder} that for any $n,c\ge 2$ the group $H^{p,p}(X_c,n)$ is generated by algebraic classes, thus in particular the surface $Z_{c,2}$ satisfies $\rho(Z_{c,2})=h^{1,1}(Z_{c,2})$. As a consequence of Proposition \ref{sing} we in fact obtain:

\begin{cor}\label{extremal} For $c\ge 2$, the surface $h:Z_{c,2}\rightarrow \mathbb{P}^1$ is an extremal elliptic surface.\end{cor}

\begin{proof}
For the fibration $h:Z_{c,2}\rightarrow \mathbb{P}^1$ and for any $v\in \mathbb{P}^1$, let $\mathcal{F}_v$ denote the fiber $h^{-1}(v)$ and let $m_v$ denote the number of components of $\mathcal{F}_v$. Define 
\[R=\{v\in \mathbb{P}^1\mid  \mathcal{F}_v \text{ is reducible} \}.\]
The Shioda-Tate formula \cite[Corollary 6.13]{elsur} expresses the Picard number $\rho(Z_{c,2})$ in terms of the reducible singular fibers and the rank $r$ of the Mordell-Weil group of $f:Z_{c,2}\rightarrow \mathbb{P}^1$:
\begin{equation}\label{stformula}\rho(Z_{c,2})=2+\sum_{v\in R}(m_v-1)+r.\end{equation}

We know from Proposition \ref{sing} that $h$ has two reducible singular fibers: one of type $I_{4\cdot3^c}$ at $0$ and one of type $I_{3^c}^*$ at $\infty$. Therefore
\[\sum_{v\in R}(m_v-1)=(4\cdot 3^c-1)+(3^c+4)=5\cdot 3^c+3.\]
So then Equation (\ref{stformula}) becomes
$\rho(Z_{c,2})=5\cdot 3^c+5+r.$

We showed in the proof of Proposition \ref{sing} that $\chi_{\mathrm{top}}(Z_{c,2})=12(g+1)=6\cdot3^c+6$. Since $h^{1,0}(Z_{c,2})=h^{0,1}(Z_{c,2})=0$ and $h^{2,0}(Z_{c,2})=h^{0,2}(Z_{c,2})=g$, it follows that $h^{1,1}(Z_{c,2})=10(g+1)=5\cdot3^c+5$.  Therefore $r=0$ and $\rho(Z_{c,2})=h^{1,1}(Z_{c,2})$.
\end{proof}

%%%%%%%%%%%%%%%%%%%%%%%%%%%%%%%%%%%%%%%%%%%%
%%%%%%%%%%%%%%%%%%%%%%%%%%%%%%%%%%%%%%%%%%%%%%%
\subsection{The $j$-invariant of $h\colon Z_{c,2}\rightarrow \mathbb{P}^1$}
In order to eventually prove that the extremal elliptic surface $h\colon Z_{c,2}\rightarrow \mathbb{P}^1$ is in fact an elliptic modular surface, it will be necessary to first describe the $j$-invariant of the fibration $h$. 

For an elliptic fibration $\varphi\colon S\rightarrow C$ without multiple fibers, consider the rational map $j\colon C\dashrightarrow \mathbb{P}^1$ given by sending each point $P\in C$ such that $\varphi^{-1}(P)$ is nonsingular to the $j$-invariant of the elliptic curve $\varphi^{-1}(P)$. This rational map $j$ can in fact be extended to all of $C$ (see for instance \cite{kodaira}). The morphism $j\colon C \rightarrow \mathbb{P}^1$ is called the $j$-\emph{invariant} of the elliptic surface $\varphi
\colon S\rightarrow C$. 

If $P\in C$ is such that $\varphi^{-1}(P)$ is singular, then we have the following (reproduced from \cite{kloosterman}):
\begin{center}
\begin{tabular}{|c|c|}
\hline
Fiber Type over $P$& $j(P)$\\
\hline
$I_0^*$ & $\ne \infty$\\
$I_b$, $I_b^*$ $(b>0)$ & $\infty$\\
$II$, $IV$, $IV^*$, $II^*$& $0$\\
$III$, $III^*$ & $1728$\\
\hline
\end{tabular}
\end{center}

\begin{lem}\label{j-inv}
For $c\ge 2$, the $j$-invariant $j\colon \mathbb{P}^1\rightarrow \mathbb{P}^1$ of $h\colon Z_{c,2}\rightarrow \mathbb{P}^1$ is non-constant.
\end{lem}
\begin{proof}
From Proposition \ref{sing}, all of the singular fibers of $h\colon Z_{c,2}\rightarrow \mathbb{P}^1$ are of type $I_b$ or $I_b^*$ with $b>0$. Hence the $j$-invariant of $h\colon  Z_{c,2}\rightarrow \mathbb{P}^1$ satisfies $j(P)=\infty$ for all $P\in \mathbb{P}^1$ such that $h^{-1}(P)$ is singular. However, since generically for $P\in \mathbb{P}^1$ the $j$-invariant $j(P)$ is the $j$-invariant of the elliptic curve $h^{-1}(P)$, generically $j$ cannot be $\infty$. Thus $j$ is non-constant.
\end{proof}

\begin{prop}\label{j-inv2} For $c\ge 2$, the $j$-invariant $j:\mathbb{P}^1\rightarrow \mathbb{P}^1$ of $h:Z_{c,2}\rightarrow \mathbb{P}^1$ has degree $6\cdot 3^c$ and is ramified at the points $0$, $1728$, and $\infty$. There are $2\cdot 3^c$ branch points above $0$, all of ramification index $3$. There are $3\cdot 3^{c}$ branch points above $1728$, all of ramification index $2$. Finally, there are $2$ branch points above $\infty$, one with ramification index $4\cdot 3^c$ corresponding to the point $0\in \mathbb{P}^1$ and one with ramification index $3^c$ corresponding to the point $\infty \in \mathbb{P}^1$.
\end{prop}

\begin{proof}
This follows directly from results of Mangala Nori in \cite{nori}.  In particular, Nori proves in \cite[Theorem 3.1]{nori} that an elliptic fibration $S\rightarrow B$ with non-constant $j$-invariant is extremal if and only if the fibration has no singular fibers of type $I_0^*$, $II$, $III$, or $IV$ and its $j$-invariant is ramified only over $0$, $1728$, and $\infty$ with ramification index $e_v$ for $v\in B$ satisfying $e_v=1$, $2$, or $3$ if $j(v)=0$ and $e_v=1$ or $2$ if $j(v)=1$.

We know from Corollary \ref{extremal} that $h\colon Z_{c,2}\rightarrow \mathbb{P}^1$ is extremal and from Lemma \ref{j-inv} that it has non-constant $j$-invariant. Hence, it follows from  \cite[Theorem 3.1]{nori} that $j:\mathbb{P}^1\rightarrow \mathbb{P}^1$ is ramified only over the points $0$, $1728$, and $\infty$. Moreover
\[\mathrm{deg}(j)=\sum_{I_b} b + \sum_{I_b^*}b,\]
where the two sums occur over all the singular fibers of $f$ of type $I_b$ and of type $I_b^*$ respectively. 

From Proposition \ref{sing}, the fibration $h\colon Z_{c,2}\rightarrow \mathbb{P}^1$ has one fiber of type $I_{4\cdot 3^c}$, one fiber of type $I_{3^c}^*$, and  $3^c$ fibers of type $I_1$. Thus $\mathrm{deg}(j)=6\cdot 3^c.$

Now let
\begin{equation*}
\begin{aligned}
\mathcal{R}_0&=\{v\in \mathbb{P}^1\mid j(v)=0\}\\
\mathcal{R}_{1728}&=\{v\in \mathbb{P}^1\mid j(v)=1728\}.
\end{aligned}
\end{equation*}

If $e_v$ denotes the ramification index of a point $v\in \mathbb{P}^1$, let
\begin{equation*}
\begin{aligned}
R_0&=\sum_{v\in \mathcal{R}_0}(e_v-1)\\
R_{1728}&=\sum_{v\in \mathcal{R}_{1728}}(e_v-1)
\end{aligned}
\end{equation*}

Then since Proposition \ref{sing} implies that $h$ has no singular fibers of type $II,$ $II^*,$ $III$, $III^*$, $IV$, or $IV^*,$ Nori's calculations in the proof of
\cite[Lemma 3.2]{nori} yield the following three equations:
\begin{equation}R_0+R_{1728}=\frac{7\cdot\mathrm{deg}(j)}{6}\end{equation}
\begin{equation}\label{r0}R_0-\frac{2\cdot\mathrm{deg}(j)}{3}\ge0 \end{equation}
\begin{equation}\label{r1}R_{1728}-\frac{\mathrm{deg}(j)}{2}\ge0 \end{equation}
Observe that
$\frac{2\cdot\mathrm{deg}(j)}{3}+ \frac{\mathrm{deg}(j)}{2}=\frac{7\cdot\mathrm{deg}(j)}{6}.$
Therefore we must have equality in Equations (\ref{r0}) and (\ref{r1}). It follows that
\begin{equation}\label{r0equation}R_0=\frac{2\cdot\mathrm{deg}(j)}{3}=4\cdot 3^c\end{equation}
\begin{equation}\label{r1equation}R_{1728}=\frac{\mathrm{deg}(j)}{2}=3\cdot3^{c}.\end{equation}
Moreover, because equality holds in  \eqref{r0}, Nori's proof in \cite[Lemma 3.2]{nori} implies that 
$\mathrm{deg}(j)=3|\mathcal{R}_0|.$
Hence we have
\begin{equation}\label{calr0}|\mathcal{R}_0|=2\cdot 3^c.\end{equation}
Now from \cite[Theorem 3.1]{nori}, for any $v\in \mathcal{R}_0$, we must have $e_v\le 3$. Hence using (\ref{calr0}), it follows that
$R_0\le 4\cdot 3^c.$
But we have already shown in Equation \eqref{r0equation} that equality holds, therefore we must have $e_v=3$ for all $v\in \mathcal{R}_0$.

Since $Z_{c,2}$ is extremal and $h:Z_{c,2}\rightarrow \mathbb{P}^1$ has no singular fibers of type $III^*$, Nori's results \cite[Theorem 3.1]{nori} also imply that $e_v=2$ for all $v\in \mathcal{R}_{1728}$.

Finally, we know $j$ has a pole of order $b_i$ at points $v_i\in \mathbb{P}^1$ where the fiber over $v_i$ is of type $I_{b_i}$ or of type $I_{b_i}^*$. Hence the result follows from Proposition \ref{sing}.
\end{proof}

%%%%%%%%%%%%%%%%%%%%%%%%%%%%%%%%%%%%%%%%%%%%%%%%%%%%%%%%%%%%
%%%%%%%%%%%%%%%%%%%%%%%%%%%%%%%%%%%%%%%%%%%%%%%%%%%%%%%%%%%%%

\subsection{Preliminaries on Elliptic Modular Surfaces} 
We begin by giving a brief introduction to elliptic modular surfaces as defined by Shioda \cite{modsur}. 
Following Nori \cite{nori}, for an elliptic surface $\varphi\colon S\rightarrow C$ with $j$-invariant $j\colon C\rightarrow \mathbb{P}^1$, let us define
\[C'=C\backslash j^{-1}\{0,1728,\infty\}.\]
In particular, for every $v\in C'$, the fiber $F_v=\varphi^{-1}(v)$ is smooth. The sheaf $G=R^1\varphi_*\mathbb{Z}$ on $C$ is the \emph{homological invariant} of the elliptic surface $S$. The restriction of $G$ to $C'$ is then a locally constant sheaf of rank two $\mathbb{Z}$-modules. Consider the monodromy homomorphism $\rho\colon \pi_1(C')\rightarrow SL(2,\mathbb{Z})$ associated to $\varphi\colon S\rightarrow C$. Observe that $\rho$ both determines and is determined by the sheaf $G$.

Conversely, let $j\colon C\rightarrow \mathbb{P}^1$ be a holomorphic map from an algebraic curve $C$ to $\mathbb{P}^1$ and let $C'=C\backslash j^{-1}\{0,1728,\infty\}.$ Let $\mathcal{H}=\{z\in \mathbb{C} \mid \mathrm{Im} (z)>0\}$ be the upper half-plane in $\mathbb{C}$ and consider the elliptic modular function $J\colon \mathcal{H}\rightarrow \mathbb{P}^1\backslash\{0,1728,\infty\}$. Finally let $U'$ be the universal cover of $C'$. Then there exists a holomorphic map $w\colon U'\rightarrow \mathcal{H}$ such that the following diagram commutes:
\begin{equation}\label{mono}
\begin{tikzcd}
U'\arrow{d}{\pi} \arrow{r}{w}&\mathcal{H}\arrow{d}{J}\\
C'\arrow{r}{j} &\mathbb{P}^1\backslash\{0,1728,\infty\}.
\end{tikzcd}
\end{equation}
This map $w$ thus induces a homomorphism $\overline{\rho}\colon \pi_1(C')\rightarrow PSL(2,\mathbb{Z})$.

Now suppose $\rho\colon \pi_1(C')\rightarrow SL(2,\mathbb{Z})$ is a homomorphism making the following diagram commute:
\[
\begin{tikzcd}[column sep=0pt]
\pi_1(C')\arrow{rd}{\overline{\rho}} \arrow{rr}{\rho}&[-6]&SL(2,\mathbb{Z})\arrow{ld}\\
&PSL(2,\mathbb{Z})&.
\end{tikzcd}
\]
Then it is possible to construct a unique elliptic surface $\varphi\colon S\rightarrow C$ having $j$-invariant given by the holomorphic map $j\colon C\rightarrow \mathbb{P}^1$ and having homological invariant given by the sheaf $G$ associated to the homomorphism $\rho$ \cite[Section 8]{kodaira}.

So now consider any finite-index subgroup $\Gamma$ of the modular group $SL(2,\mathbb{Z})$ not containing $-\mathrm{Id}$. Then $\Gamma$ acts on the upper half plane $\mathcal{H}$ and the quotient $\Gamma \backslash \mathcal{H}$, together with a finite number of cusps, forms an algebraic curve $C_\Gamma$. For any other such subgroup $\Gamma'$, if $\Gamma\subset \Gamma'$, then the canonical map $\Gamma \backslash \mathcal{H}\rightarrow \Gamma' \backslash \mathcal{H}$ extends to a holomorphic map $C_\Gamma \rightarrow C_{\Gamma'}$. In particular, taking $\Gamma'=SL(2,\mathbb{Z})$ and identifying $C_{\Gamma'}$ with $\mathbb{P}^1$ via the elliptic modular function $J$, we get a holomorphic map
\[j_{\Gamma}\colon C_{\Gamma}\rightarrow \mathbb{P}^1.\]
Hence, as discussed, there exists a $w\colon U'\rightarrow \mathcal{H}$ fitting into a diagram (\ref{mono}) which induces a representation $\overline{\rho}\colon \pi_1(C')\rightarrow \overline{\Gamma}\subset PSL(2,\mathbb{Z})$, where $\overline{\Gamma}$ is the image of $\Gamma$ in $PSL(2,\mathbb{Z})$. Because $\Gamma$ contains no element of order $2$, this homomorphism $\overline{\rho}$ lifts to a homomorphism $\rho\colon \pi_1(C')\rightarrow SL(2,\mathbb{Z})$, which then gives rise to a sheaf $G_\Gamma$ on $C_\Gamma$.

\begin{defn}\cite{modsur} For any finite index subgroup $\Gamma$ of $SL(2,\mathbb{Z})$ not containing $-\mathrm{Id}$, the associated elliptic surface $\varphi\colon S_\Gamma\rightarrow C_\Gamma$ having $j$-invariant $j_\Gamma$ and homological invariant $G_\Gamma$ is called the \emph{elliptic modular surface} attached to $\Gamma$.  
\end{defn}

%%%%%%%%%%%%%%%%%%%%%%%%%%%%%%%%%%%%%%%%%%%%%%%%%%%%%%%%%%%%
%%%%%%%%%%%%%%%%%%%%%%%%%%%%%%%%%%%%%%%%%%%%%%%%%%%%%%%%%%%%%
\subsection{The surface $Z_{c,2}$ is elliptic modular}
We now return to considering the elliptic surface $h\colon Z_{c,2}\rightarrow \mathbb{P}^1$. Let us define the following elements $A_0, A_1 \ldots, A_{3^c}, A_\infty$ of $SL(2,\mathbb{Z})$ as elements of the following conjugacy classes:
\[
A_0\in \left[\left(
\begin{array}{cc}
1&4\cdot 3^c\\
0& 1
\end{array}\right)\right]
\ \ \ \ \ \ \ 
A_1,\ldots, A_{3^c}\in \left[\left(
\begin{array}{cc}
1&1\\
0& 1
\end{array}\right)\right]
\ \ \ \ \ \ \ 
A_\infty\in \left[\left(
\begin{array}{cc}
-1&-3^c\\
0& -1
\end{array}
\right)
\right]
\]

Then consider the subgroup $\Gamma_c$ of index $6\cdot 3^c$ in $SL(2,\mathbb{Z})$ with the following presentation:
\[\Gamma_c:=\langle A_0, A_1,\ldots, A_{3^c}, A_{\infty} \mid A_0A_1\cdots A_{3^c}A_\infty=\mathrm{Id}\rangle.\]
We remark that $\Gamma_c$ is not a congruence subgroup as it does not appear on the list in \cite{experimental} of the genus $0$ congruence subgroups of $SL(2,\mathbb{Z})$ (see \cite{rademacher} for more details on such subgroups).

\begin{thm}\label{elmod} For $c\ge 2$, the surface $Z_{c,2}$ is the elliptic modular surface attached to $\Gamma_c$. \end{thm}
\begin{proof}
In \cite[Theorem 3.5]{nori}, Mangala Nori proves that  an extremal elliptic surface $\varphi\colon S\rightarrow C$ with a section and with non-constant $j$-invariant is an elliptic modular surface as long as $\varphi\colon S\rightarrow C$ has no singular fibers of type $II^*$ or $III^*$ in Kodaira's classification. Therefore, since the surface $h\colon Z_{c,2}\rightarrow \mathbb{P}^1$ is extremal  (by Corollary \ref{extremal}), has a section (by Proposition \ref{sing}), has non-constant $j$-invariant (by Lemma \ref{j-inv}), and only has fibers of type $I_b$ and $I_b^*$ (by Proposition \ref{sing}), we know $Z_{c,2}$  is indeed an elliptic modular surface. 

So let $\Gamma$ be the finite-index subgroup of $SL(2,\mathbb{Z})$ attached to $Z_{c,2}$. By Proposition \ref{j-inv2}, the degree of the $j$-invariant of $h:Z_{c,2}\rightarrow \mathbb{P}^1$ is $6\cdot 3^c$. Hence the group $\Gamma$ has index $6\cdot 3^c$ in $SL(2,\mathbb{Z})$. 

Now consider the $j$-invariant $j:\mathbb{P}^1\rightarrow \mathbb{P}^1$ of $Z_{c,2}$, which we have investigated in Proposition \ref{j-inv2}. Let $C'=\mathbb{P}^1\backslash \{0,1728,\infty\}$.

Because $Z_{c,2}$ is elliptic modular, its $j$-invariant induces a homomorphism 
\[\rho\colon \pi_1(C')\rightarrow \Gamma\subset SL(2,\mathbb{Z}).\]
Let us write the set
\[j^{-1}\{0,1728,\infty\}=\{v_1,\ldots,v_s\}.\]
By Proposition  \ref{j-inv2}, we know $s=5\cdot 3^c +2$. For each point $v_i$ let $\alpha_i$ be the loop element in $\pi_1(C')$ going around $v_i$. 
Then $\pi_1(C')$ is the free group on these generators $\alpha_1,\ldots, \alpha_s$ subject to the relation (taken in cyclic order) $\alpha_1\cdots\alpha_s=1$ \cite[Lemma 2.1]{bogomolov}.

In \cite[Proposition 1.4]{nori}, Nori describes, for an elliptic surface $S\rightarrow C$ with loop elements $\alpha_i\in \pi_1(C)$ around $v_i\in C$, the possible values of $\rho(\alpha_i)$ depending on the values of $j(v_i)$. 
By Proposition \ref{j-inv2}, for our surface $h\colon Z_{c,2}\rightarrow \mathbb{P}^1$ all of the points $v_i$ such that $j(v_i)=0$ have ramification index $3$. Hence by \cite[Proposition 1.4]{nori}, for the corresponding $\alpha_i$, we have $\rho(\alpha_i)= \pm \mathrm{Id}$. However since $Z_{c,2}$ is elliptic modular, the subgroup $\Gamma$ cannot contain $-\mathrm{Id}$. Hence, for all $i$ such that $j(v_i)=0$, we must have $\rho(\alpha_i) =\mathrm{Id}$.  

Similarly, by Proposition \ref{j-inv2} all of the points $v_i$ such that $j(v_i)=1728$ have ramification index $2$. But then by \cite[Proposition 1.4]{nori}, for all such $i$, we have $\rho(\alpha_i)=\pm \mathrm{Id}$ and thus, in fact, $\rho(\alpha_i)=  \mathrm{Id}$.

Therefore the only points $v_i\in j^{-1}\{0,1728,\infty\}$ that contribute non-identity elements to $\Gamma$ are the points sent to $\infty$ by $j$. These are exactly the points of $\mathbb{P}^1$ underneath the singular fibers of $h\colon Z_{c,2}\rightarrow \mathbb{P}^1$.

From \cite[Proposition 4.2]{modsur}, if a point $v_i$ has singular fiber of type $I_b$ with $b>0$, then
\[\rho(\alpha_i) \in \left[ \left( \begin{array}{cc} 1&b\\ 0 & 1\end{array}\right)\right].\]
If a point $v_i$ has singular fiber of type $I_b^*$ with $b>0$, then 
\[\rho(\alpha_i) \in \left[ \left( \begin{array}{cc} -1&-b\\ 0 & -1\end{array}\right)\right].\]
Therefore, using Proposition \ref{sing}, in the case of $h\colon Z_{c,2}\rightarrow \mathbb{P}^1$, the point $0$ contributes a generator $A_0$ of $\Gamma$ in the conjugacy class of 
\[\left( \begin{array}{cc} 1&4\cdot 3^c\\ 0 & 1\end{array}\right)\]
in $SL(2,\mathbb{Z})$. Each point $\zeta^i$, for $\zeta$ a $3^c$-th root of unity, contributes a generator $A_{i+1}$ in the conjugacy class of 
\[\left( \begin{array}{cc} 1&1\\ 0 & 1\end{array}\right).\]
Finally, the point $\infty$ contributes a generator $A_\infty$ in the conjugacy class of
\[\left( \begin{array}{cc} -1&-3^c\\ 0 & -1\end{array}\right).\]
 
Then $\Gamma$ is the free group on these generators $A_0, A_1, \ldots, A_{3^c}, A_\infty$ subject to the relation $$A_0 A_1 \cdots A_{3^c}A_\infty=\mathrm{Id}.$$

Hence we indeed have that $\Gamma$ is the group $\Gamma_c$ defined above.

\end{proof}

%%%%%%%%%%%%%%%%%%%%%%%%%%%%%%%%%%%%%%%%%%%%%%%%%%%%%%%%%%%%
%%%%%%%%%%%%%%%%%%%%%%%%%%%%%%%%%%%%%%%%%%%%%%%%%%%%%%%%%%%%%

\section{The threefold case}\label{threefold section}
We now consider the case of the threefold $X_{c,3}$, which by construction is a smooth model of the quotient $\overline{X}_{c,3}\coloneqq C_g\times Z_{c,2}/\langle \psi_g^{-1}\times \phi_{c,2}\rangle$. 
For $m$ sufficiently divisible the rational map  
\[f\colon C_{g}^3/G\dashrightarrow \mathbb{P}^1\subset  \mathbb{P}(H^0(C_{g}^3,K_{C_{g}^3}^{\otimes m}))^G\]
studied in Proposition \ref{fibration} can be resolved to obtain a morphism
\[\tilde{f}\colon Z_{c,3}\rightarrow \mathbb{P}^1\subset \mathbb{P}(H^0(Z_{c,3},K_{Z_{c,3}}^{\otimes m})).\]
In this threefold case, since the abundance conjecture is known, we may assume that $Z_{c,3}$ is a minimal model of $X_{c,3}$. In order to study the fibration $\tilde{f}$, we will need some information about the minimal model $Z_{c,2}$ and its elliptic fibration studied in Section \ref{surface section}. 

%%%%%%%%%%%%%%%%%%%%%%%%%%%%%%%%%%%%%%%%%%%%%%%%%%%%%%%%%%%%
%%%%%%%%%%%%%%%%%%%%%%%%%%%%%%%%%%%%%%%%%%%%%%%%%%%%%%%%%%%%%

\subsection{The rational map $\beta\colon C_g\times Z_{c,2} \dashrightarrow \mathbb{P}(H^0(\overline{X}_{c,3},K_{\overline{X}_{c,3}}^{\otimes m}))$}
We now consider the product $C_g\times Z_{c,2}$. As in \eqref{maindiagram}, we have a diagram
\begin{equation}\label{betadiagram}
\begin{tikzcd}
\overline{X}_{c,3}\rar[-,dashrightarrow]&\mathbb{P}(H^0(\overline{X}_{c,3},K_{\overline{X}_{c,3}}^{\otimes m}))\\
C_g\times Z_{c,2}\uar[-,rightarrow]\arrow{r} &\mathbb{P}(H^0(C_g\times Z_{c,2},K_{C_g\times Z_{c,2}}^{\otimes m}))\uar[-,rightarrow]
\end{tikzcd}
\end{equation}
Here the horizontal maps are the Iitaka fibrations and the vertical map on the left is the quotient by the group $\langle \psi_g^{-1}\times \phi_{c,2}\rangle$.

We then consider the composition 
\[\beta\colon C_g\times Z_{c,2} \dashrightarrow \mathbb{P}(H^0(\overline{X}_{c,3},K_{\overline{X}_{c,3}}^{\otimes m})).\]

Recall from Theorem \ref{global} that global sections of $K_{C_g}^{\otimes m}$ are of the form 
$x^a\omega$  for  $0\le a\le m(g-1).$ Moreover, from the proof of Proposition \ref{koddim}, global sections $s_a$  of $K_{Z_{c,2}}^{\otimes m}$ for $0\le a \le m(g-1)$ correspond to global sections of $K_{C_g^2}^{\otimes m}$ of the form
$x_1^a\omega_1\times x_2^a\omega_2.$ Hence we may view $\beta\colon C_g\times Z_{c,2} \dashrightarrow \mathbb{P}(H^0(\overline{X}_{c,3},K_{\overline{X}_{c,3}}^{\otimes m}))$ as the rational map given by
\begin{equation}\label{beta map}
(z_1,z_2,z_3)\mapsto [\omega_1(z_1)\times s_0(z_2,z_3): (x_1\omega_1)(z_1)\times s_1(z_2,z_3):\cdots: (x_1^{m(g-1)}\omega_1)(z_1)\times s_{m(g-1)}(z_2,z_3)]
\end{equation}

%%%%%%%%%%%%%%%%%%%%%%%%%%%%%%%%%%%%%%%%%%%%%%%%%%%%%%%%%%%%
%%%%%%%%%%%%%%%%%%%%%%%%%%%%%%%%%%%%%%%%%%%%%%%%%%%%%%%%%%%%%

\subsubsection{The images of $C_g\times T$ and $C_g\times S$}\label{beta S and T section} Our main interest here will be the images under $\beta$ of the surfaces of the form $C_g\times T$ and $C_g\times S$ in $C_g\times Z_{c,2}$, where recall from Section \ref{hirzebruchjung section} that $T$ and $S$ are the rational curves with self-intersection $-2$ and $-(g+1)$ respectively in the resolution of the Type II singular points of $C_g\times C_g/\langle \psi_g^{-1}\times \psi_g\rangle$. We provide a pictorial representation of the main facts about these images in Figure \ref{curves picture} below, where one color represents regions sent to $0$, one color represents regions sent to $\infty$, and one color represents regions where $\beta$ is undefined. 

%%%%%%%%%%%%%%%%%%%%%%%%%%%%%%%%%%%%%%%%%%%%%%%%%%%%%%%%%%%%
%%%%%%%%%%%%%%%%%%%%%%%%%%%%%%%%%%%%%%%%%%%%%%%%%%%%%%%%%%%%%

\begin{figure}
\centering
\caption{Images of $C_g\times S$ and $C_g\times T$ under $\beta$. }
%Note that $\beta$ is undefined on the curves $P_i\times T$ and the point $Q\times \sigma(0)$}
\label{curves picture}
\begin{tikzpicture}
\draw (0,0) rectangle (4,4)node[below=70pt, pos=0.5]{$S$}node[left=70pt, pos=0.5]{$C_g$};
\node [below] at (0.25,0) {$\sigma(0)$};
\node [below] at (3.75,0) {$\sigma(\infty)$};
\node [left] at (0,.25) {$P_i$};
\node [left] at (0,3.75) {$Q$};
\draw[red, line width=.1cm] (0,4)--(0,-0.05);
\draw[red, line width=.1cm](0,0)--(4,0);
\draw[blue, line width=.1cm] (0,4)--(4,4);
\fill[orange!30!yellow] (0,4) circle (.1cm);
\draw (4,0) rectangle (8,4)node[below=70pt, pos=0.5]{$T$};
\fill[blue] (4,0) rectangle (8,4);
\draw[blue, line width=.1cm] (4,4)--(8,4);
\draw[orange!30!yellow, line width=.1cm] (4,0)--(8,0);
\draw[thick,->] (4,-1)--(4,-3) node[pos=.5, left]{$\beta$};
\draw (0,-4)node[align=right,   below]{\\$0$}--(8,-4)node[align=right,   below]{\\$\infty$};
\fill[blue] (8,-4) circle (.1cm);
\fill[red] (0,-4) circle (.1cm);
\node at (4,-5){$\mathbb{P}^1$};
\end{tikzpicture}
\end{figure}

%%%%%%%%%%%%%%%%%%%%%%%%%%%%%%%%%%%%%%%%%%%%%%%%%%%%%%%%%%%%
%%%%%%%%%%%%%%%%%%%%%%%%%%%%%%%%%%%%%%%%%%%%%%%%%%%%%%%%%%%%%

Begin by considering the elliptic fibration $h\colon Z_{c,2}\rightarrow \mathbb{P}^1$ from Section \ref{surface section}. As discussed in the proof of Proposition \ref{sing}, the map $h$ sends each curve $T$ to the point $\infty$ in $\mathbb{P}^1$, while the curves $S$ corresponds to sections of $h$. In particular, there exists a map $\sigma\colon \mathbb{P}^1\rightarrow S$ such that $h\circ \sigma =\mathrm{Id}_{\mathbb{P}^1}$. In particular, for local coordinates $(z_2,z_3)$ on $S$, we have $s_0(z_2,z_3)=0$ if and only if $(z_2,z_3)=\sigma(\infty)$ and $s_{m(g-1)}(z_2,z_3)=0$ if and only if $(z_2,z_3)=\sigma(0)$. 

Moreover, as discussed in the proof of Proposition \ref{fibration}, the only points of $C_g$ on which the form $x_1^a\omega_1$ can vanish are the points $P_1$, $P_2$ and $Q$. More precisely, as established in \eqref{E1}, \eqref{E2}, \eqref{E3}, we have $x_1^a\omega_1(P_i)=0$ if and only if  $1\le a\le m(g-1)$ and $x_1^a\omega_1(Q)=0$ if and only if $0\le a\le m(g-1)-1$. 

Let us first consider the surfaces of the form $C_g\times T$ in $C_g\times Z_{c,2}$. We know that the fibration $h\colon Z_{c,2}\rightarrow \mathbb{P}^1$ sends the curve $T$ to the point $\infty$, meaning that for local coordinates $(z_2,z_3)\in T$ we have $s_{m(g-1)}(z_2,z_3)\ne0$ and $s_a(z_2,z_3)=0$ for all other $1\le a\le m(g-1)-1$. In particular, using the description of $\beta$ in \eqref{beta map}, for any point $W\in C_g$ such that $W\ne P_i$ for $i\in \{1,2\}$, we have $\beta(W\times T)=[0:\cdots:0:1]$, which corresponds to the point $\infty\in \mathbb{P}^1\subset \mathbb{P}(H^0(\overline{X}_{c,3},K_{\overline{X}_{c,3}}^{\otimes m}))$. In other words, we have $\beta((C_g-\{P_1,P_2\})\times T)=\infty$ and $\beta$ is not defined on the curves $P_1\times T$ and $P_2\times T$. 

Now consider the surfaces of the form $C_g\times S$ in $C_g\times Z_{c,2}$. Let $U=C_g-\{P_1,P_2,Q\}$. Then since $x_1^a\omega_1$ does not vanish on $U$ for any $0\le a\le m(g-1)$ and $f_2$ sends $S$ surjectively onto $\mathbb{P}^1$, we have that $\beta$ sends $U\times S$ surjectively onto $\mathbb{P}^1$. In particular, we have $\beta(U\times \sigma(0))=0$ and $\beta(U\times \sigma(\infty))=\infty$. 

It thus remains to determine what happens to the curves of the form $P_1\times S$, $P_2\times S$, and $Q\times S$. However we know that for local coordinates $(z_2,z_3)$ on $S$, we have $s_0(z_2,z_3)=0$ if and only if $(z_2,z_3)=\sigma(\infty)$ and $s_{m(g-1)}(z_2,z_3)=0$ if and only if $(z_2,z_3)=\sigma(0)$. It follows that $\beta(P_i\times (S-\sigma(\infty))=[1:0:\cdots:0]$, which corresponds to the point $0\in \mathbb{P}^1\subset \mathbb{P}(H^0(\overline{X}_{c,3},K_{\overline{X}_{c,3}}^{\otimes m}))$, and $\beta(Q\times (S-\sigma(0))=[0:\cdots:0:1]$, which corresponds to the point $\infty\in \mathbb{P}^1$. Hence the only points of $C_g\times S$ on which $\beta$ is not defined are the points $Q\times \sigma(0)$ and $P_i\times \sigma(\infty)$ for $i\in \{1,2\}$. 

Note however that since the curves $S$ and $T$ intersect at the point $\sigma(\infty) \in S$, the assertion that $\beta$ is not defined at the points $P_i\times \sigma(\infty)$ is already covered by the previous assertion that $\beta$ is not defined on the curves $P_i\times T$.  

In summary, we have established the following about the rational map $\beta\colon C_g\times Z_{c,2} \dashrightarrow \mathbb{P}^1$:
\begin{enumerate}
\item $\beta((C_g-\{P_1,P_2\})\times T)=\infty$
\item $\beta$ sends $(C_g-\{P_1,P_2,Q\})\times S$ surjectively onto $\mathbb{P}^1$ and in particular 
\begin{enumerate}
\item $\beta(U\times \sigma(0))=0$ 
\item $\beta(U\times \sigma(\infty))=\infty$
\end{enumerate}
\item $\beta(P_i\times (S-\sigma(\infty))=0$ for $i\in \{1,2\}$
\item $\beta(Q\times (S-\sigma(0))=\infty$
\item $\beta$ is undefined on the curves $P_1\times T$ and $P_2\times T$ and on the point $Q\times \sigma(0)$
\end{enumerate}

\subsection{The K3 fibration  $\tilde{f}\colon Z_{c,3}\rightarrow \mathbb{P}^1$}
We will now use this understanding of the rational map $\beta$ in order to show that the fibration $\tilde{f}\colon Z_{c,3}\rightarrow \mathbb{P}^1$ is indeed a K3 fibration. 

\begin{thm}\label{main theorem}
For $c\ge 2$ and $m$ sufficiently divisible, the general smooth fibers of the fibration 
\[\tilde{f}\colon Z_{c,3}\rightarrow \mathbb{P}^1\subset \mathbb{P}(H^0(Z_{c,3},K_{Z_{c,3}}^{\otimes m}))\] are K3 surfaces with Picard rank $19$. Moreover, the fibers of $\tilde{f}$ 
 above the $3^c$ roots of $t^{3^c}+1$ are (possibly singular) K3 surfaces with Picard rank $20$. 
\end{thm}

\begin{proof}
For $t\in \mathbb{P}^1$, let $\mathcal{F}_t$ be a general smooth fiber of the fibration $\tilde{f}\colon Z_{c,3}\rightarrow \mathbb{P}^1$. In order to show that $\mathcal{F}_t$ is a K3 surface of Picard rank $19$ we will identify a configuration of $21$ curves on $\mathcal{F}_t$ whose intersection matrix has rank $19$. Since the Picard rank $\rho$ of a general smooth fiber of $\tilde{f}$ can be at most $19$, this yields that $\rho(F_t)=19$. 

We will identify these $21$ curves on $\mathcal{F}_t$ by first identifying $21$ hypersurfaces in $Z_{c,3}$ that surject onto $\mathbb{P}^1$ via $\tilde{f}$. Begin by considering the diagram \eqref{betadiagram} and, using the notation of  Section \ref{beta S and T section}, two hypersurfaces of the form $C_g\times T$ and $C_g\times S$ in $C_g\times Z_{c,2}$.  As in Section \ref{beta S and T section}, let $\sigma\colon \mathbb{P}^1\rightarrow S$ denote the section corresponding to the curve $S$ of the ellitpic fibration $f_2\colon Z_{c,2}\rightarrow \mathbb{P}^1$. Hence the surfaces of the form $C_g\times T$ and $C_g\times S$ intersect along the curve $C_g\times \sigma(\infty)$. 

Recall the calculations of Section \ref{beta S and T section} about the images of the surfaces $C_g\times T$ and $C_g\times S$ under the rational map
\[\beta\colon C_g\times Z_{c,2} \rightarrow \mathbb{P}(H^0(C_g\times Z_{c,2},K_{C_g\times Z_{c,2}}^{\otimes m})\dashrightarrow \mathbb{P}^1\subset \mathbb{P}(H^0(\overline{X}_{c,3},K_{\overline{X}_{c,3}}^{\otimes m})).\]
In particular, recall that we established that the open subset 
\[(C_g\times S)-\{P_1\times \sigma(0), P_2\times \sigma(0), Q\times \sigma(\infty)\}\subset C_g\times S\]
is sent surjectively onto $\mathbb{P}^1$ by $\beta$. Consider the image of this open subset in $Z_{c,3}$ and let $\mathcal{G}$ denote its closure. Then by the diagram \eqref{betadiagram} and since $\tilde{f}\colon Z_{c,3}\rightarrow \mathbb{P}^1$ is a morphism, we have that $\mathcal{G}\subset Z_{c,3}$ surjects onto $\mathbb{P}^1$ via $\tilde{f}$.

Consider the curve on the fiber $\mathcal{F}_t$ given by $\gamma=\mathcal{F}_t\cap \mathcal{G}$. Recall that the curve $S$ in $Z_{c,2}$ is obtained from a point of the form $P_i\times Q$ in $C_g^2$. In particular, the surface $\mathcal{G}$ lies in the image in $Z_{c,3}$ of the curve $C_g\times P_i\times Q$ in $C_g^3$. As discussed in the proof of Proposition \ref{fibration}, the fiber $\mathcal{F}_t$ arises as the image in $Z_{c,3}$ of the hypersurface $F_t$ in $C_g^3$ cut out by the equation $x_1x_2x_3=t$. Note that the intersection of $F_t$ and the curve $C_g\times P_i\times Q$ has dimension $0$. Hence the curve $\gamma$ is contracted to a point in the singular variety $C_g^3/G$. 

Now note that from the construction of the fibration $\tilde{f}$ as the Iitaka fibration of $Z_{c,3}$ we know that the fiber $\mathcal{F}_t$ has Kodaira dimension $0$. In particular, from the classification of algebraic surfaces, we know that $K_{\mathcal{F}_t}$ is torsion. It follows from the adjunction formula that because $\gamma$ is a curve on $\mathcal{F}_t$ that can be contracted to a point, the curve $\gamma$ must be rational with self-intersection $-2$. 

Hence we have identified a rational $(-2)$-curve on the fiber $\mathcal{F}_t$ arising from a curve of the form $C_g\times P_i\times Q$ in $C_g^3$. Symmetrically, any permutation of these factors will also yield a $(-2)$-curve on $\mathcal{F}_t$. There are thus $6$ such curves for each choice of $i\in \{1,2\}$, meaning there are $12$ such curves total on $\mathcal{F}_t$. 

Now recall from Section \ref{beta S and T section} that the rational map $\beta$ is not defined on curves of the form $P_1\times T$ and $P_2\times T$ and on the point $Q\times \sigma(0)$. In order to understand the images of these in $Z_{c,3}$, we would like to understand their action by the automorphism $\psi_g^{-1}\times \phi_{c,2}$. Recall the description of the local action of $\phi_{c,2}$ on the curves $T$ and $S$ established in \eqref{phi weights} of Section \ref{S and T phi}. Note in particular that the curve $T$ lies in the fixed locus of $\phi_{c,2}$. Hence, the curves $P_1\times T$ and $P_2\times T$ lie in the fixed locus of $\psi_g^{-1}\times \phi_{c,2}$. Blowing-up this part of the fixed locus, will then yield exceptional divisors of codimension $2$ in $Z_{c,3}$. Namely, the images in $Z_{c,3}$ of the curves of the form $P_1\times T$ and $P_2\times T$ will intersect $\mathcal{F}_t$ in at most points and so will not contribute to the Picard rank of $\mathcal{F}_t$.

It thus remains to consider the image in $Z_{c,3}$ of a point of the form $Q\times \sigma(0)$. From the description of the local action of $\phi_{c,2}$ on $S$ given in \eqref{phi weights} of Section \ref{S and T phi}, we have that $\psi_g^{-1}\times \phi_{c,2}$ acts locally around the point $Q\times \sigma(0)$ with $\mathbb{Z}/3^c\mathbb{Z}$-weights $(-g,1,g)$. Note that this implies that $Q\times \sigma(0)$ is an isolated terminal singularity in the quotient $(C_g\times Z_{c,2})/\langle \psi_g^{-1}\times \phi_{c,2}\rangle$ \cite[Remark 2.5(i)]{morrison}. In particular, the resolution of this point in $Z_{c,3}$ has codimension $1$. Moreover, we established in Section \ref{beta S and T section} that the rational map $\beta$ sends the curve $(C_g-Q)\times \sigma(0)$ to the point $0\in \mathbb{P}^1$ and sends the curve $Q\times (S-\sigma(0))$ to the point $\infty\in \mathbb{P}^1$. Hence by continuity since the morphism $\tilde{f}\colon Z_{c,3}\rightarrow \mathbb{P}^1$ is defined everywhere, the resolution of $Q\times \sigma(0)$ in $Z_{c,3}$ surjects onto $\mathbb{P}^1$ via $f$. In particular, this resolution contains some hypersurface $\mathcal{H}$ surjecting onto $\mathbb{P}^1$ via $\tilde{f}$.

Consider the curve on the fiber $\mathcal{F}_t$ given by $\delta=\mathcal{F}_t\cap \mathcal{H}$. Observe that since the hypersurface $\mathcal{H}$ can be contracted to a point, so can the curve $\delta$. Hence as in the case of the curve $\gamma$, because $K_{\mathcal{F}_t}$ is torsion, the adjunction formula yields that $\delta$ is a rational $(-2)$-curve on $\mathcal{F}_t$. 

Hence we have identified a rational $(-2)$-curve on $\mathcal{F}_t$ arising from a point of the form $Q\times P_i\times Q$ in $C_g^3$. Symmetrically, any permutation of these factors will also yield a $(-2)$-curve on $\mathcal{F}_t$. There are thus $3$ such curves for each choice of $i\in \{1,2\}$, meaning there are $6$ such curves total on $\mathcal{F}_t$. 

Thus for the moment we have identified $18$ rational $(-2)$-curves on $\mathcal{F}_t$, where for each choice of $i\in \{1,2\}$, there are $9$ such curves: six of the form $\gamma$ and three of the form $\delta$. 
We now determine the configuration of these $18$ curves. 

For a given choice of $i\in \{1,2\}$, let $\gamma_1,\ldots,\gamma_6$ denote the six $(-2)$-curves on $F_t$ associated respectively to the curves $C_g\times P_i\times Q$, $P_i\times C_g\times Q$, $P_i\times Q\times C_g$, $C_g\times Q\times P_i$, $Q\times C_g \times P_i$, and $Q\times P_i\times C_g$ in $C_g^3$. Let $\delta_1$, $\delta_2$, and $\delta_3$ denote the three $(-2)$-curves on $F_t$ associated respectively to the points $P_i\times Q\times Q$, $Q\times Q\times P_i$, and  $Q\times P_i\times Q$ in $C_g^3$. Hence $\gamma_1$ meets $\gamma_2$, which possibly meets $\delta_1$, which possibly meets $\gamma_3$, which then meets $\gamma_4$, which possibly meets $\delta_2$, which possibly meets $\gamma_5$, which then meets $\gamma_6$, which possibly meets $\delta_3$, which possibly meets $\gamma_1$. Let $M$ denote the configuration matrix of these $9$ curves $\gamma_1$, $\gamma_2$, $\delta_1$, $\gamma_3$, $\gamma_4$, $\delta_2$, $\gamma_5$, $\gamma_6$. If all of the $\delta_j$ meet their neighboring $\gamma_k$, then these $9$ curves form a cycle which has intersection matrix the rank-$8$ circulant matrix
\[\left(\begin{array}{ccccccccc}
-2&1&0&0&0&0&0&0&1\\
1&-2&1&0&0&0&0&0&0\\
0&1&-2&1&0&0&0&0&0\\
0&0&1&-2&1&0&0&0&0\\
0&0&0&1&-2&1&0&0&0\\
0&0&0&0&1&-2&1&0&0\\
0&0&0&0&0&1&-2&1&0\\
0&0&0&0&0&0&1&-2&1\\
1&0&0&0&0&0&0&1&-2
\end{array}
\right).
\]
Note that if any of the $\delta_j$ fail to meet a neighboring $\gamma_k$, then the resulting intersection matrix $M$ has rank $9$. Namely, we have $\mathrm{rank}(M)\ge 8$. Note that this already yields that the Picard rank $\rho$ of the fiber $\mathcal{F}_t$ is at least $16$. Since we know the surface $\mathcal{F}_t$ has Kodaira dimension $0$, it follows from the classification of algebraic surfaces that $\mathcal{F}_t$ is a K3 surface. 

We now identify three additional curves on the smooth fiber $\mathcal{F}_t$. Consider the three classes in $\mathrm{Pic}(C_g^3)$ given by $[\mathrm{pt}\times C_g\times C_g]$, $[C_g\times \mathrm{pt}\times C_g]$, and $[ C_g\times C_g\times \mathrm{pt}]$. Note that these classes are $G$-invariant and thus each descends to yield a hypersurface class in $Z_{c,3}$. Denote these three hypersurface classes by $[D_1]$, $[D_2]$, $[D_3]$. Moreover, a general hypersurface $H\in [D_i]$ intersects the smooth fiber $\mathcal{F}_t$ in a curve $\varepsilon_j$. Note that the curve $\varepsilon_j$ will not intersect any of the previously identified $(-2)$-curves of the form $\gamma$ or $\delta$. Moreover since  the fiber $\mathcal{F}_t$ arises as the image in $Z_{c,3}$ of the hypersurface $F_t$ in $C_g^3$ cut out by the equation $x_1x_2x_3=t$, it follows that $\varepsilon_j.\varepsilon_k=2$ where the two intersection points correspond to the two points of $C_g$ with the same $x$-value.  Thus we need only compute the self-intersections $\varepsilon_j^2$ for $j\in \{1,2,3\}$.

To compute these self-intersections $\varepsilon_j^2$, let us without loss of generality consider the curve $\varepsilon_1=H\cap \mathcal{F}_t$ for $H$ a general hypersurface in $[D_1]$, where $[D_1]$ arises from the class $[\mathrm{pt}\times C_g\times C_g]$ in $\mathrm{Pic}(C_g^3)$.  Let $\widetilde{H}\in [\mathrm{pt}\times C_g\times C_g]$ be in the preimage of $H$ in $C_g^3$ and let $\widetilde{\varepsilon}_1=F_t\cap \widetilde{H}$ be in the preimage of $\varepsilon_1$, where recall that $\mathcal{F}_t$ arises as the image in $Z_{c,3}$ of the hypersurface $F_t$ in $C_g^3$. Let $\overline{\varepsilon}_1$ be the image of $\widetilde{\varepsilon}_1$ in $C_g\times Z_{c,2}$. Namely, we have a sequence of $\mathbb{Z}/3^c\mathbb{Z}$-quotients given by $\widetilde{\varepsilon}_1\rightarrow \overline{\varepsilon}_1\rightarrow \varepsilon_1$.

Suppose that for $\widetilde{H}\in [\mathrm{pt}\times C_g\times C_g]$ the chosen point in the first factor is given by local coordinates $(x_1,y_1)=(\alpha_0,\beta_0)$. Then the curve  $\widetilde{\varepsilon}_1$ in $C_g^3$ is cut out by the equations $x_1=\alpha_0$ and $x_2x_3=\frac{t}{\alpha_0}$ (note that we may assume $\alpha_0\ne 0$). This yields an isomorphism between the image curve $\overline{\varepsilon}_1$ in $C_g\times Z_{c,2}$ and the fiber of the two-dimensional Iitaka fibration $f_2\colon Z_{c,2}\rightarrow \mathbb{P}^1$ above the point $\frac{t}{\alpha_0}\in \mathbb{P}^1$. In particular, the curve $\overline{\varepsilon}_1$ has genus $1$. It follows that the curve $\varepsilon_1$ has genus at most $1$. 

However, since the curve $\varepsilon_1$ lies on the surface $\mathcal{F}_t$, which we have established is a K3 surface, the adjunction formula yields
\[K_{ \varepsilon_1}=(K_{\mathcal{F}_t}+H)|_H=H|_H.\]
In particular, we have $2g(\varepsilon_1)-2\ge 0$, meaning that $g(\varepsilon_1)\ge 1$. Hence the curve $\varepsilon_1$ has genus equal to $1$. Applying the adjunction formula on a K3 surface again then yields that $\varepsilon_1$ has self-intersection $0$.  An identical argument yields that the same is true of the curves $\varepsilon_2$ and $\varepsilon_3$.

Therefore, the intersection matrix of the configuration of three curves $\varepsilon_1$, $\varepsilon_2$, and $\varepsilon_3$ is the rank $3$ matrix 
\[N=\left(\begin{array}{ccc}
0 & 2&2\\
2&0&2\\
2&2&0
\end{array}
\right)
.\]

Namely, we have identified a configuration of $21$ curves on a general  smooth fiber $\mathcal{F}_t$ of the fibration $\tilde{f}\colon Z_{c,3}\rightarrow \mathbb{P}^1$ having intersection matrix 
\[I=\left(\begin{array}{ccc}
M & 0&0\\
0&M&0\\
0&0&N
\end{array}
\right)
.\]
In particular, the matrix $I$ has rank at least $8+8+3=19$, so the Picard rank $\rho$ of $\mathcal{F}_t$ is at least $19$. However, since the fiber $\mathcal{F}_t$ varies with $t$, we know that $\rho(\mathcal{F}_t)$ is at most $19$, meaning we have $\rho(\mathcal{F}_t)=19$.

The statement about the fibers of $\tilde{f}$ above the roots of $t^{3^c}+1$ follows from Proposition \ref{fibration} and its proof. Indeed, we know that a fiber of $f\colon C_{g}^3/G\dashrightarrow \mathbb{P}^1\subset  \mathbb{P}(H^0(C_{g}^3,K_{C_{g}^3}^{\otimes m}))^G$ above a point $t\in \mathbb{P}^1$ such that $t^{3^c}+1$ has an isolated singular point at the image in $C_{g}^3/G$ of a point in $C_g^3$ of the form  $((x_1, 0), (x_2, 0), (x_3, 0))$, where the $x_j$ are of the form $\xi^{2\gamma_j+1}$ for $\xi$ a primitive $2\cdot 3^c$-th root of unity and $0\le \gamma_j\le 3^c-1$. In particular, $x_j\ne \pm 1$ and so while the corresponding point in  $C_g^3/G$ is a singularity of the fiber of $f$, it is not a singular point of $C_g^3/G$. Thus depending on the birational model $Z_{c,3}$, such a point either remains a singular point of the fiber of $\tilde{f}$ above $\zeta^i$ or aquires a curve resolving the singularity. Since the fibers of $\tilde{f}$ above $\zeta^i$ contain the configuration of $21$ curves identified above, it follows that the (after possible resolving the singular point), these fibers have Picard number $20$.

\end{proof}
%%%%%%%%%%%%%%%%%%%%%%%%%%%%%%%%%%%%%%%%%%%%%%%%%%%%%%%%%%%%%

\subsection{Moduli of K3 surfaces}
Here we follow \cite[Section 2.3]{sarti}. Recall that if $S$ is an algebraic K3 surface then the group $H^2(S,\mathbb{Z})$ together with its intersection pairing has the structure of a unimodular lattice isometric to the K3 lattice 
\[\Lambda\coloneqq E_8(-1)^{\oplus 2}\oplus U^{\oplus 3},\] where $E_8$ is the unique postive-definite, even, unimodular lattice of rank $8$ and $U$ is the hyperbolic plane given by 
$U\coloneqq \left( \begin{array}{cc}0&1\\1&0 \end{array}\right)$.  Moreover, there is a surjective period map to the coarse moduli space given by the quadric in $\mathbb{P}^{21}$
\[\Omega\coloneqq \{ [\omega]\in \mathbb{P}(\Lambda\otimes \mathbb{C}\mid (\omega,\omega)=0,\ (\omega,\overline{\omega})>0\}.\]
Namely if $\omega_S\in H^{2,0}(S)$, then $\omega_S$ yields a class in $\Omega$ and conversely every point in $\Omega$ is the period point of some K3 surface. If $L\subset \Lambda$ is some sublattice of signature $(1,\rho-1)$, then the subspace
\[\Omega_L\coloneqq \{[\omega]\in \Omega\mid (\omega,\lambda)=0 \text{ for all } \lambda\in L\}\]
has dimension $20-\rho=20-\mathrm{rank}(L)$. 

If $L$ is the Neron-Severi group $NS(S)\coloneqq H^2(S,\mathbb{Z})\cap H^{1,1}(S)$ and $\rho=19$, then the embedding of the transcendental lattice $T_S\coloneqq NS(S)^{\perp}$ in $\Lambda$ is unique up to an isometry of $\Lambda$ \cite[Corollary 2.10]{morrison2} and the moduli curve $\Omega_L$ is determined by $T_S$. 

Note moreover that in this case, the CM points of the curve $\Omega_L$ correspond to singular K3 surfaces, meaning K3 surfaces with Picard rank $20$ whose Neron-Severi group contains the Neron-Severi group of a general member of the curve $\Omega_L$. 

We thus obtain the following corollary from Theorem \ref{main theorem}.

\begin{cor}\label{final cor}
The one-dimensional family $\tilde{f}\colon Z_{c,3}\rightarrow \mathbb{P}^1$ of K3 surfaces  of Picard rank $19$ is a finite cover of the universal family of the moduli curve $\Omega_{\mathcal{F}_t}$ isomorphic to  $\mathbb{P}^1$ parametrizing K3 surfaces with transcendental lattice $T_{\mathcal{F}_t}$, where $\mathcal{F}_t$ denotes a general fiber of $\tilde{f}$, and  the fibers of $\tilde{f}$ at the $3^c$ roots of the polynomial  $t^{3^c}+1$ correspond to CM points of $\Omega_{\mathcal{F}_t}$. 
\end{cor}

%%%%%%%%%%%%%%%%%%%%%%%%%%%%%%%%%%%%%%%%%%%%%%%%%%%%%%%%%%%%%

\textbf{Acknowledgements.} The author would like to thank Jaclyn Lang, Christopher Lyons, Matt Kerr, Emanuele Macr\`{i}, Luca Schaffler, Stefan Schreieder, Burt Totaro, and Charles Vial for helpful discussions and comments in the preparation of this article. 
The author gratefully acknowledges support of the National Science Foundation through awards DGE-1144087 and DMS-1645877.

\bibliography{Schreieder.bib}

\begin{thebibliography}{BHPVdV15}

\bibitem[BHPVdV15]{BHPV}
W.~Barth, K.~Hulek, C.~Peters, and A.~Van~de Ven.
\newblock {\em Compact complex surfaces}, volume~4.
\newblock Springer, 2015.

\bibitem[BT03]{bogomolov}
F.~Bogomolov and Y.~Tschinkel.
\newblock Monodromy of elliptic surfaces.
\newblock In {\em {G}alois groups and fundamental groups}, volume~41 of {\em
  Math. Sci. Res. Inst. Publ.}, pages 167--181. {C}ambridge {U}niv. {P}ress,
  {C}ambridge, 2003.

\bibitem[BV04]{BV}
A.~Beauville and C.~Voisin.
\newblock On the {C}how ring of a {K}$3$ surface.
\newblock {\em J. Alg. Geom.}, 13:417--426, 2004.

\bibitem[CD89]{cossec}
F.~Cossec and I.~Dolgachev.
\newblock {\em {E}nriques surfaces. {I}}, volume~76 of {\em Progress in
  Mathematics}.
\newblock Birkh{\"a}user Boston, Inc., Boston, MA, 1989.

\bibitem[CH07]{cynk}
S.~Cynk and K.~Hulek.
\newblock Higher-dimensional modular {C}alabi-{Y}au manifolds.
\newblock {\em Canad. Math. Bull.}, 50(4):486--503, 2007.

\bibitem[CLY04]{rademacher}
K.S. Chua, M.~L. Lang, and Y.~Yang.
\newblock On {R}ademacher's conjecture: congruence subgroups of genus zero of
  the modular group.
\newblock {\em J. Algebra}, 277(1):408--428, 2004.

\bibitem[CP03]{experimental}
C.~J. Cummins and S.~Pauli.
\newblock Congruence subgroups of {${\rm PSL}(2,{\Bbb Z})$} of genus less than
  or equal to 24.
\newblock {\em Experiment. Math.}, 12(2):243--255, 2003.

\bibitem[FL18]{FlapanLang}
L.~Flapan and J.~Lang.
\newblock Chow motives associated to certain algebraic {H}ecke characters.
\newblock {\em Trans. Amer. Math. Soc. Ser. B}, 5:102--124, 2018.

\bibitem[Klo04]{kloosterman}
R.~Kloosterman.
\newblock Extremal elliptic surfaces and infinitesimal {T}orelli.
\newblock {\em Michigan Math. J.}, 52(1):141--161, 2004.

\bibitem[Kod60]{kodaira2}
K.~Kodaira.
\newblock On compact analytic surfaces.
\newblock In {\em {A}nalytic functions}, pages 121--135. {P}rinceton Univ.
  Press, Princeton, N.J., 1960.

\bibitem[Kod63]{kodaira}
K.~Kodaira.
\newblock On compact analytic surfaces. {II}, {III}.
\newblock {\em Ann. of Math. (2) 77 (1963), 563--626; ibid.}, 78:1--40, 1963.

\bibitem[Kol07]{kollar}
J.~Koll{\'a}r.
\newblock {\em Resolution of singularities}.
\newblock Princeton University Press, 2007.

\bibitem[KT15]{kock}
B.~K{\"o}ck and J.~Tait.
\newblock Faithfulness of actions on {R}iemann-{R}och spaces.
\newblock {\em Canad. J. Math.}, 67(4):848--869, 2015.

\bibitem[Laz04]{positivity}
R.~Lazarsfeld.
\newblock {\em Positivity in algebraic geometry. {I}}.
\newblock Springer-Verlag, Berlin, 2004.

\bibitem[LV17]{LV}
R.~Laterveer and C.~Vial.
\newblock On the {C}how ring of {C}ynk-{H}ulek {C}alabi-{Y}au varieties and
  {S}chreieder varieties.
\newblock \url{https://arxiv.org/abs/1712.03070}, 2017.

\bibitem[Mir89]{miranda}
R.~Miranda.
\newblock {\em The basic theory of elliptic surfaces}.
\newblock Dottorato di Ricerca in Matematica. [Doctorate in Mathematical
  Research]. ETS Editrice, Pisa, 1989.

\bibitem[Mor84]{morrison2}
D.~R. Morrison.
\newblock On {$K3$}\ surfaces with large {P}icard number.
\newblock {\em Invent. Math.}, 75(1):105--121, 1984.

\bibitem[MS84]{morrison}
D.~Morrison and G.~Stevens.
\newblock Terminal quotient singularities in dimensions three and four.
\newblock {\em Proc. Amer. Math. Soc.}, 90(1):15--20, 1984.

\bibitem[Nik79]{nikulin}
V.~V. Nikulin.
\newblock Integer symmetric bilinear forms and some of their geometric
  applications.
\newblock {\em Izv. Akad. Nauk SSSR Ser. Mat.}, 43(1):111--177, 238, 1979.

\bibitem[Nor85]{nori}
M.~Nori.
\newblock On certain elliptic surfaces with maximal {P}icard number.
\newblock {\em Topology}, 24(2):175--186, 1985.

\bibitem[Rei12]{reid}
M.~Reid.
\newblock Surface cyclic quotient singularities and {H}irzebruch-{J}ung
  resolutions.
\newblock 2012.
\newblock \url{http://www. warwick. ac. uk/masda/surf}.

\bibitem[Sar08]{sarti}
A.~Sarti.
\newblock Transcendental lattices of some {$K3$}-surfaces.
\newblock {\em Math. Nachr.}, 281(7):1031--1046, 2008.

\bibitem[Sch15]{schreieder}
S.~Schreieder.
\newblock On the construction problem for {H}odge numbers.
\newblock {\em Geom. Topol.}, 19(1):295--342, 2015.

\bibitem[Shi72]{modsur}
T.~Shioda.
\newblock On elliptic modular surfaces.
\newblock {\em J. Math. Soc. Japan}, 24:20--59, 1972.

\bibitem[SS10]{elsur}
M.~Sch{{\"u}}tt and T.~Shioda.
\newblock Elliptic surfaces.
\newblock In {\em Algebraic Geometry in {E}ast {A}sia---{S}eoul 2008},
  volume~60 of {\em Adv. Stud. Pure Math.}, pages 51--160. Math. Soc. Japan,
  Tokyo, 2010.

\end{thebibliography}
\bibliographystyle{alpha}

\end{document}